\documentclass[a4paper,11pt,leqno]{article}
\usepackage{amsmath,amssymb}
\usepackage{amsfonts}
\usepackage{eucal}
\usepackage{amsthm}
\numberwithin{equation}{section}

\newcommand{\bigpare}[1]{\bigl(#1\bigr)}
\newcommand{\biggpare}[1]{\biggl(#1\biggr)}
\newcommand{\Bigpare}[1]{\Bigl(#1\Bigr)}

\newcommand{\bigbra}[1]{\bigl\{#1\bigr\}}

\newcommand{\biggbra}[1]{\biggl\{#1\biggr\}}

\newcommand{\bigbrac}[1]{\bigl[#1\bigr]}

\newcommand{\bigset}[2]{\bigl\{#1\bigm|#2\bigr\}}

\newcommand{\norm}[1]{\| #1 \|}
\newcommand{\bignorm}[1]{\bigl\| #1 \bigr\|}
\newcommand{\Bignorm}[1]{\Bigl\| #1 \Bigr\|}

\newcommand{\bigabs}[1]{\bigl| #1 \bigr|}

\newcommand{\biggabs}[1]{\biggl| #1 \biggr|}

\newcommand{\jap}[1]{\langle #1 \rangle}

\def\a{\alpha}
\def\b{\beta}
\def\c{\gamma}
\def\d{\delta}
\def\e{\varepsilon}
\def\f{\varphi}
\def\g{\psi}

\def\l{\lambda}
\def\m{\mu}
\def\n{\nu}
\def\o{\omega}
\def\s{\sigma}
\def\t{\tau}
\def\x{\xi}
\def\y{\eta}

\def\th{\theta}

\newcommand{\G}{\Psi}
\newcommand{\I}{\mathcal{I}}
\renewcommand{\L}{\Lambda}
\renewcommand{\O}{\Omega}
\renewcommand{\S}{\Sigma}

\def\re{\mathbb{R}}

\def\ze{\mathbb{Z}}

\def\pa{\partial}

\newcommand{\supp}{\mathrm{{supp}}}

\newcommand{\dist}{\mathrm{dist}}

\DeclareMathOperator*{\slim}{s-lim}
\DeclareMathOperator*{\wlim}{w-lim}

\newcommand{\Ran}{\mathrm{Ran\;}}

\newtheorem{thm}{Theorem}[section]
\newtheorem{lem}[thm]{Lemma}
\newtheorem{prop}[thm]{Proposition}
\newtheorem{cor}[thm]{Corollary}

\theoremstyle{definition}

\newtheorem{ass}{Assumption}

\theoremstyle{remark}
\newtheorem{rem}{Remark}[section]

\numberwithin{equation}{section}

\title{Microlocal properties of scattering matrices for 
Schr\"odinger equations on scattering manifolds}
\author{
Kenichi I{\sc to}%
\footnote{Graduate School of Pure and Applied Sciences, University of Tsukuba,
1-1-1 Tennodai, Tsukuba Ibaraki, 305-8571 Japan. 
E-mail: \texttt{ito\_ken@math.tsukuba.ac.jp}. Partially supported by JSPS Research Fellowship (2006-2007); JSPS Grant Wakate (B) 21740090 (2009--2012) }
\ and Shu N{\sc akamura}%
\footnote{Graduate School of Mathematical Sciences, 
University of Tokyo, 3-8-1 Komaba, Meguro Tokyo, 
153-8914 Japan. 
E-mail: {\tt shu@ms.u-tokyo.ac.jp}.  
Partially supported by JSPS Grant Kiban (B) 17340033 (2005--2008); Kiban (A) 
21244008 (2009-2013).} }

\begin{document}
\maketitle

\begin{abstract}
Let $M$ be a scattering manifold, i.e., a Riemannian manifold with asymptotically conic structure, 
and let $H$ be a Schr\"odinger operator on $M$. We can construct a natural time-dependent 
scattering theory for $H$ with a suitable reference system, and the scattering matrix is defined 
accordingly (\cite{IN2}). We here show the scattering matrices are Fourier integral operators 
associated to a canonical transform on the boundary manifold
generated by the geodesic flow. 
In particular, we learn that the wave front sets are mapped according to the canonical transform. 
These results are generalizations of a theorem by Melrose and Zworski \cite{MZ}, but the 
framework and the proof are quite different. 
These results may be considered as generalizations or refinements of the classical off-diagonal 
smoothness of the scattering matrix for 2-body quantum scattering on Euclidean spaces. 
\end{abstract}


\section{Introduction}

Let $M$ be an $n$-dimensional smooth non-compact manifold such that 
$M=M_c\cup M_\infty$, where $M_c$ is relatively compact and $M_\infty$ is 
diffeomorphic to $\re_+\times \pa M$ with a compact manifold $\pa M$. 
In the following, we often identify $M_\infty$ with $\re_+\times \pa M$, and 
we also suppose $M_c\cap M_\infty \subset (0,1)\times \pa M$ under this 
identification. 

We first recall the construction of the model introduced in \cite{IN2}. 
Let $\{\f_\a\,:\,U_\a\to \re^{n-1}\}$, $U_\a\subset \pa M$, be a local coordinate 
system of $\pa M$. We set 
\[
\{\tilde \f_\a=I\otimes \f_\a\,:\, \tilde U_\a =\re_+\times U_\a \to \re\times \re^{n-1}\}
\]
be a local coordinate system of $M_\infty\cong \re_+\times \pa M$, and we use  
$(r,\th)\in \re\times \re^{n-1}$ to represent a point in $M_\infty$. 

We suppose $\pa M$ is equipped with a smooth strictly positive density $H=H(\th)$ and a 
positive $(2,0)$-tensor $h=(h^{jk}(\th))$ on $\pa M$. We let
\[
Q=-\frac12 \sum_{j,k} H(\th)^{-1} \frac{\pa}{\pa \th_j} H(\th) h^{jk}(\th)\frac{\pa}{\pa \th_k}
\quad \text{on } \mathcal{H}_b = L^2(\pa M, H(\th)d\th). 
\]
$Q$ is an essentially self-adjoint operator on $\mathcal{H}_b$, and we denote the unique 
self-adjoint extension by the same symbol $Q$. 

We set $G$ be a smooth strictly positive density on $M$ such that 
\[
G(x)dx = r^{n-1}H(\th) drd\th \quad \text{on } (1,\infty)\times \pa M\subset M_\infty, 
\]
and we set $\mathcal{H}=L^2(M, G(x)dx)$. Let $P$ be a formally self-adjoint second order 
elliptic operator on $M$ such that 
\[
P=-\frac12 G^{-1} (\pa_r ,\pa_\th/r) G 
\begin{pmatrix} a_1 & a_2 \\ {}^t a_2 & a_3 \end{pmatrix}
\begin{pmatrix}\pa _r \\ \pa_\th/r \end{pmatrix} +V
\quad \text{on } M_\infty,
\]
where $\begin{pmatrix} a_1 & a_2 \\ {}^t a_2 & a_3 \end{pmatrix}$ defines a real-valued smooth tensor, and $V$ is a real-valued smooth function. 
We suppose, as well as in \cite{IN2}, 

\begin{ass}\label{ass:a}
There is $\m>0$ such that for any $\ell\in \ze_+$, $\a\in \ze_+^{n-1}$ there is $C_{\ell\a}>0$ and 
\begin{align*}
&\bigabs{\pa_r^\ell\pa_\th^\a (a_1(r,\th)-1)}\leq C_{\ell\a} r^{-1-\m-\ell},\quad 
\bigabs{\pa_r^\ell\pa_\th^\a a_2(r,\th)}\leq C_{\ell\a} r^{-\m-\ell},\\
&\bigabs{\pa_r^\ell\pa_\th^\a (a_3(r,\th)-h(\th))}\leq C_{\ell\a} r^{-\m-\ell}, 
\quad \bigabs{\pa_r^\ell\pa_\th^\a V(r,\th)}\leq C_{\ell\a} r^{-1-\m-\ell}
\end{align*}
in each local coordinate of $M_\infty$ described above. 
\end{ass}

We may consider $P$ as a {\it short range}\/ perturbation of $-\frac12\pa_r^2+\frac{1}{r^2}Q$, 
but we will use different operators to construct a scattering theory. 
It is known that $P$ is essentially self-adjoint; $\s_{ess}(P)=[0,\infty)$; and $P$ is absolutely 
continuous except for a countable discrete spectrum with the only possible accumulation point 0 
(see \cite{IN2} and references therein). We construct a time-dependent scattering theory for 
$H$ as follows: We set 
\begin{align*}
&M_f = \re\times \pa M, \quad \mathcal{H}_f =L^2(M_f, H(\th)drd\th), \\
& P_f = -\frac12 \frac{\pa^2}{\pa r^2} \quad \text{on } M_f. 
\end{align*}
$P_f$ is the one-dimensional free Schr\"odinger operator, and it is self-adjoint with 
$\mathcal{D}(P_f)= H^2(\re)\otimes \mathcal{H}_b$. Let $j(r)\in C^\infty(\re)$ such that 
$j(r)=0$ on $(-\infty, \frac12]$ and $j(r)=1$ on $[1,\infty)$. We set $\mathcal{I} : 
\mathcal{H}_f \to \mathcal{H}$ by 
\[
(\mathcal{I}\f)(r,\th) = r^{-(n-1)/2} j(r) \f(r,\th) \quad \text{if } (r,\th)\in M_\infty, 
\]
and $(\mathcal{I}\f)(x)=0$ if $x\notin M_\infty$. We define the wave operators by 
\[
W_\pm =W_\pm(P,P_f,\mathcal{I}) =\slim_{t\to\pm\infty} e^{itP} \mathcal{I} e^{-itP_f}.
\]
It is shown in \cite{IN2} that $W_\pm$ exist, and are complete in the following sense. 
Let $\mathcal{F}$ be the Fourier transform in $r$, i.e., 
\[
(\mathcal{F}\f)(\rho,\th) =(2\pi)^{-1/2} \int_{-\infty}^\infty e^{-ir\rho} \f(r,\th)dr
\quad \text{for } \f\in C_0^\infty(M_f), 
\]
and it is extended to a unitary map in $L^2(M_f)$. If we set 
\[
\mathcal{H}_{f,\pm} =\bigset{\f\in\mathcal{H}_f}{\supp(\mathcal{F}\f)\subset \overline{\re_\pm}
 \times \pa M},
\]
then $\mathcal{H}_f=\mathcal{H}_{f,+}\oplus\mathcal{H}_{f,-}$. We consider $W_\pm$ as 
maps from $\mathcal{H}_{f,\pm}$ to $\mathcal{H}$. Then $W_\pm$ are asymptotically 
complete, i.e., $W_\pm$ are unitary operators from $\mathcal{H}_{f,\pm}$ to 
$\mathcal{H}_{ac}(P)$ (\cite{IN2}, Theorem~2). Then the scattering operator defined by 
\[
S=W_+^* W_- \ :\ \mathcal{H}_{f,-}\to \mathcal{H}_{f,+}
\]
is a unitary operator. By the intertwining property: $P_f S= S P_f$, there is 
$S(\l)\in \mathcal{B}(\mathcal{H}_b)$ for $\l>0$ such that 
\[
(\mathcal{F} S \mathcal{F}^{-1}\f )(\rho,\cdot) = S(\rho^2/2) \f(-\rho,\cdot)
\quad \text{for }\rho>0, \f\in \mathcal{F} \mathcal{H}_{f,-}.
\]
$S(\l)$ is our scattering matrix, and we study its microlocal properties. 

Let 
\[
q(\th,\o)=\frac12 \sum_{j,k} h^{jk}(\th)\o_j\o_k \quad \text{for }(\th,\o)\in T^*\pa M
\]
be the classical Hamiltonian associated to $Q$. We denote the Hamilton flow generated by $b$ 
by $\exp(tH_b)$ for $t\in\re$. 

\begin{thm}\label{thm:WF}
Suppose Assumption~\ref{ass:a}, and let $u\in \mathcal{H}_b$. Then 
\[
W\!F(S(\l)u)= \exp(\pi H_{\sqrt{2q}})W\!F(u),
\]
where $W\!F(u)$ denotes the wave front set of $u$. 
\end{thm}

If $\m=1$, then we can show $S(\l)$ is an FIO. This is a slight
extension of a theorem by Melrose and Zworski \cite{MZ}. 

\begin{thm}\label{thm:MZ}
Suppose Assumption~\ref{ass:a} with $\m=1$. Then for each $\l>0$, 
$S(\l)$ is an FIO associated to $\exp(\pi H_{\sqrt{2q}})$. 
\end{thm}

If $0<\m<1$, then $S(\l)$ is not necessarily an FIO in the usual sense, but we can 
still show it is an FIO in a generalized sense: 

\begin{thm}\label{thm:GFIO}
Suppose Assumption~\ref{ass:a}, and let $S(\l)$ be the scattering matrix defined as above. 
Then for each $\l>0$, $S(\l)$ is a Fourier integral operator associated to an asymptotically 
homogeneous canonical transform in $T^* \pa M$, which is asymptotic to 
$\exp (\pi H_{\sqrt{2q}})$ as $\o \to\infty$. 
\end{thm}

The exact definition of the term: {\it an FIO associated to asymptotically homogeneous 
canonical transform}\/ is given in \cite{IN3}, and we discuss it in Section~6. 

\begin{rem}
Since we do not introduce a Riemannian metric, our model looks rather different from the scattering 
metric defined by Melrose \cite{Me,MZ}. However, as is explained in Appendix of \cite{IN2}, 
the Laplacian on scattering manifolds is a special case of our model. Namely, their model 
corresponds to the case that $\m=1$ and that each $a_j$ has asymptotic expansion in $r^{-1}$ 
as $r \to \infty$, and $V=0$. 
\end{rem}

Theorems~\ref{thm:WF} and \ref{thm:MZ} are essentially corollaries of Theorem~\ref{thm:GFIO}, 
but they can be proved by a simpler argument than Theorem~\ref{thm:GFIO}. 
We feel the simpler argument is interesting in itself, and we first prove Theorems~\ref{thm:WF} 
and \ref{thm:MZ}, and then we refine the argument to prove Theorem~\ref{thm:GFIO} later. 

The main idea to prove Theorems 1.1--1.3 is to consider the evolution:
\[
A(t)= e^{itP_f/h} \mathcal{I}^* e^{-itP/h} a(hr, D_r, \th, h D_\th) e^{itP/h} \mathcal{I} e^{-itP_f/h}
\]
with some symbol $a$, and use an Egorov theorem type argument for this time-dependent
operator. We use semiclassical argument, i.e., we consider the asymptotic behavior of the operator 
as $h\to 0$. We consider $W(t)=e^{itP_f/h} \mathcal{I}^* e^{-itP/h} $ as a time-evolution, 
and then construct an asymptotic solution for $A(t)$ (with slight modifications) as a solution to a 
Heisenberg equation. 
The construction of the asymptotic solution relies on the classical Hamilton flow 
generated by $p$, the symbol of $P$. The dominant part of the symbol $p$ is given by the 
unperturbed conic Hamiltonian: $p_c=\frac12\rho^2+\frac{1}{r^2}q(\th,\o)$. 
The classical scattering operator for the pair 
$p_c$ and $p_f=\tfrac12 \rho^2$ is explicitly computed, and it is $\exp(\pi H_{\sqrt{2q}})$, 
which appears in the statement of our main theorems. Thus, one may consider our results 
as a quantization of the classical mechanical scattering on the scattering manifold. 
More precisely, we show that the canonical transform appearing in Theorem~1.3 is actually 
the classical scattering map for the pair $p$ and $p_f$, which is not necessarily 
homogeneous, and we need to use the method of FIOs with asymptotically conic  
Lagrangian manifolds. 

As mentioned in the beginning, Theorem~1.2 is slight generalization of the Melrose-Zworski 
Theorem (\cite{MZ}. See also \cite{Va} for a simplification of the theory). 
They used the theory of Legendre distribution and the notion of scattering wave front sets, 
whereas we use relatively elementary pseudo\-differential operator calculus with 
somewhat non-standard symbol classes, and a Beals type characterization of FIOs. 
We also note that our proof, as well as the setting, is time-dependent theoretical, and we investigate 
the scattering phenomena directly to obtain the properties of the wave operators and scattering operators, whereas the Melrose-Zworski paper relies on the stationary, generalized eigenfunction expansion theory. 

Our method is closely related to our previous works on the propagation of singularities for 
Schr\"odinger evolution equations (\cite{Na1, Na2, IN1, IN3}). In these works, we considered 
singularities of solutions, which is described by their high energy behavior, whereas in the scattering 
phenomena we are concerned with the large $r$ behavior (which in turn is related to the high $|\o|$ behavior, 
where $\o$ is the conjugate variables to  $\th\in \pa M$).  Thus we are forced to use 
different symbol classes in the calculus, and the corresponding classical mechanics looks slightly 
different, but the general strategy is essentially the same as these papers. 

If $M=\re^n$ and the Hamiltonian $P$ is a short-range perturbation of the Laplacian $-\frac12\triangle$, 
then the canonical map $\exp(\pi H_{\sqrt{2q}})$ is the antipodal map on $T^*S^{n-1}$. 
In this case, the off-diagonal smoothness of the scattering cross section is well-known 
(see \cite{IK1}, \cite{Ya}, Section 9.4, and references therein), and our result (as well as 
the Melrose-Zworski theorem) may be considered as its generalizations. 
For such models, our result implies the scattering matrix is an FIO (associated to a canonical 
map which is asymptotic to the identity map), and 
if $\m=1$ then it is in fact a pseudodifferential operator. It is also not difficult to show 
from our argument that the scattering matrix is a pseudodifferential operator 
with the symbol in $S^0_{\m, 0}(S^{n-1})$ if $\m\in (0,1)$. 

The paper is constructed as follows. In Section~2, we discuss Hamilton flows generated by 
$p_c$ and $p$, and their scattering theory. In Section~3, we prepare symbol calculus on the 
scattering manifolds. In Section~4, we discuss an Egorov type theorem and the construction of asymptotic 
solutions, which is sufficient to show Theorems~1.1 and 1.2. We prove Theorems~1.1 and 1.2 
in Section~5. In Section~6, we discuss the modification of the argument to show Theorem~1.3. 
We discuss a local decay estimate necessary in the proof in Appendix~A. A Beals type 
characterization, or an inverse of the Egorov theorem is discussed in Appendix~B, along with a 
technical lemma on FIOs used in the proof. 

Throughout this paper, we use the following notation: 
For norm spaces $X$ and $Y$, the space of bounded linear maps is denoted by 
$B(X,Y)$, and if $X=Y$, we also denote $B(X,X)=B(X)$. More generally, 
if $X$ and $Y$ are topological linear spaces, the space of continuous linear maps 
is denoted by $\mathcal{L}(X,Y)$. For a symbol $g$ on $T^*X$ with a manifold $X$, 
we denote the Hamilton flow generated by the Hamilton vector field:
\[
H_g=\frac{\pa g}{\pa \x}\cdot \frac{\pa}{\pa x} -\frac{\pa g}{\pa x}\cdot \frac{\pa}{\pa \x} 
\]
by $\exp(tH_g)$. We also denote $T^*X\setminus 0 =\bigset{(x,\x)\in T^*X}{\x\neq 0}$.


\section{Classical flow and scattering theory}

In this section, we consider the classical mechanics, or the Hamilton flow 
for the Hamiltonian with conic structure on $T^*M_\infty$ where $M_\infty=\re_+\times \pa M$, 
and then the Hamilton flow generated by the principal symbol of $P$.

\subsection{Exact solutions to the conic Hamilton flow}

We set
\[
p_c(r,\rho, \th,\o) =\frac12\rho^2+\frac{1}{r^2}q(\th,\o), \quad
q(\th,\o) = \frac12 \sum_{j,k} h^{jk}(\th)\o_j\o_k
\]
on $T^*M_\infty \cong T^*\re_+\times T^*\pa M$. We consider 
\[
(r(t),\rho(t),\th(t),\o(t))= \exp(tH_{p_c})(r_0,\rho_0,\th_0,\o_0)
\]
with $(r_0,\rho_0,\th_0,\o_0)\in T^*\re_+\times (T^*\pa M\setminus 0)$, 
i.e., $\o_0\neq 0$. It satisfies the Hamilton equation:
\begin{align*}
&r'(t) =\frac{\pa p_c}{\pa \rho}=\rho(t), \quad 
\rho'(t)=-\frac{\pa p_c}{\pa r} =\frac{2}{r(t)^3} q(\th(t),\o(t)), \\
&\th'(t)=\frac{\pa p_c}{\pa \o} =\frac{1}{r(t)^2}\frac{\pa q}{\pa \o}(\th(t),\o(t)), \quad
\o'(t) =\frac{\pa p_c}{\pa \th} =-\frac{1}{r(t)^2}\frac{\pa q}{\pa \th}(\th(t),\o(t)).
\end{align*}
The solution has two invariants: the total energy $E_0=p_c(r_0,\rho_0,\th_0,\o_0)$ 
and the angular energy $q_0=q(\th_0,\o_0)$. (The conservation of the total energy follows 
from $\{p_c,p_c\}=0$, and the angular energy from $\{q,p_c\}=\frac12\{q,\rho^2\} 
+\{q,\frac{1}{r^2}\} q +\frac{1}{r^2}\{q,q\}=0$). Then $(r(t),\rho(t))$ satisfies 
\[
r'(t)=\rho(t),\quad \rho'(t)=\frac{2}{r(t)^3}q_0
\]
which is independent of $(\th(t),\o(t))$. Noting $(r^2(t))''=4E_0$, we can easily solve 
this equation to obtain
\[
r(t)=\sqrt{2E_0t^2+2r_0\rho_0 t+r_0^2}, \quad 
\rho(t)=\frac{2E_0t+r_0\rho_0}{\sqrt{2E_0t^2+2r_0\rho_0 t+r_0^2}}, \quad t\in\re.
\]
We now set 
\[
\t(t)=\int_0^t \frac{ds}{r(s)^2} =\frac{1}{\sqrt{2q_0}}\biggbra{
\tan^{-1}\biggpare{\frac{2E_0t+r_0\rho_0}{\sqrt{2q_0}}} 
-\tan^{-1}\biggpare{\frac{r_0\rho_0}{\sqrt{2 q_0}}}}.
\]
Then $(\th(t),\o(t))$ satisfies 
\[
\frac{\pa \th}{\pa \t} = \frac{\pa q}{\pa \o}(\th,\o), \quad 
\frac{\pa \o}{\pa \t} =-\frac{\pa q}{\pa \th}(\th,\o).
\]
and hence we learn 
\[
(\th(t),\o(t))= \exp(\t(t) H_q)(\th_0,\o_0).
\]
Moreover, if we set $\s(t)=\sqrt{2q_0}\cdot \t(t)$, then we learn
\[
\frac{\pa \th}{\pa \s} = \frac{1}{\sqrt{2q}}\frac{\pa q}{\pa \o}=\frac{\pa \sqrt{2q}}{\pa \o}(\th,\o), \quad 
\frac{\pa \o}{\pa \s} =-\frac{1}{\sqrt{2q}}\frac{\pa q}{\pa \th}=-\frac{\pa\sqrt{2q}}{\pa \th}(\th,\o), 
\]
and hence 
\[
(\th(t),\o(t))= \exp\bigpare{\s(t) H_{\sqrt{2q}}}(\th_0,\o_0). 
\]
Note that $\exp\bigpare{t H_{\sqrt{2q}}}$ is the geodesic flow on $\pa M$ with respect 
to the (co)metric $(h^{jk}(\th))$ on $T^* \pa M$. 

\subsection{Classical mechanical wave operators and scattering operator for the conic Hamilton flow}

Now we consider the asymptotics as $t\to\pm\infty$. We set 
\begin{align*}
&r_\pm = \lim_{t\to\pm\infty} \tilde r(t)=\lim_{t\to\pm\infty} (r(t)-t\rho(t)) = \pm \frac{r_0\rho_0}{\sqrt{2E_0}}, \\
&\rho_\pm = \lim_{t\to\pm\infty} \rho(t)=\pm \sqrt{2 E_0}, \\
&(\th_\pm,\o_\pm)=\lim_{t\to\pm\infty} (\th(t),\o(t)) =\exp(\s_\pm H_{\sqrt{2q}})(\th_0,\o_0),
\end{align*}
where $\s_\pm = \pm \frac{\pi}{2} -\tan^{-1}\bigpare{\frac{r_0\rho_0}{\sqrt{2q_0}}}$. 
Note we need a modification only for $r(t)$. $(r_\pm,\rho_\pm,\th_\pm,\o_\pm)$ are 
the scattering data for the trajectory $(r(t),\rho(t),\th(t),\o(t))$. 
We also note the identities: 
\[
E_0=\frac12\rho_0^2+\frac{1}{r_0^2}q_0 = \frac12 \rho_\pm^2, \quad 
r_0\rho_0=r_\pm\rho_\pm, \quad
q_0=q(\th_\pm,\o_\pm). 
\]
Using these, we can solve $(r_0,\rho_0,\th_0,\o_0)$ for given 
$(r_\pm,\rho_\pm,\th_\pm,\o_\pm)$ if $\pm \rho_\pm>0$ and $\o_\pm\neq 0$: 
\begin{align*}
&r_0=\sqrt{r_\pm^2+2q_0/\rho_\pm^2}, \quad 
\rho_0=\frac{r_\pm\rho_\pm}{\sqrt{r_\pm^2+2q_0/\rho_\pm^2}}, \\
&(\th_0,\o_0) = \exp(-\s_\pm H_{\sqrt{2q}})(\th_\pm,\o_\pm),
\end{align*}
where $\s_\pm =\pm \frac{\pi}{2} -\tan^{-1}\bigpare{\frac{r_\pm\rho_\pm}{\sqrt{2q}}}$. 
We define the classical wave operators (for the pair $p_c$ and $p_f:=\frac12\rho^2$) by 
\[
w_{c,\pm} : (r_\pm,\rho_\pm,\th_\pm,\o_\pm) \mapsto (r_0,\rho_0,\th_0,\o_0). 
\]
We can also write 
\[
w_{c,\pm}(r_\pm,\rho_\pm,\th_\pm,\o_\pm)=\lim_{t\to\pm\infty}
\exp(-tH_{p_c})\circ \exp(tH_{p_f})(r_\pm,\rho_\pm,\th_\pm,\o_\pm).
\]
It is easy to check $w_{c,\pm}$ are diffeomorphisms from 
$\re\times \re_\pm \times (T^*\pa M\setminus 0)$ 
to $\re_+\times\re\times (T^*\pa M\setminus 0)$. Hence the classical scattering operator:
\[
s_c =w_{c,+}^{-1}\circ w_{c,-} : 
(r_-,\rho_-,\th_-,\o_-) \mapsto (r_+,\rho_+,\th_+,\o_+)
\]
is a diffeomorphism from $\re\times \re_-\times (T^*\pa M\setminus 0)$ to 
$\re\times \re_+\times (T^*\pa M\setminus 0)$. 
We can easily compute $s_c$ explicitly, and we have
\[
s_c(r,\rho,\th,\o)= (-r,-\rho,\exp(\pi H_{\sqrt{2q}})(\th,\rho)), 
\]
and this is the classical analogue of the Melrose-Zworski theorem. 

We write
\[
w_c(t)= \exp(-tH_{p_c})\circ \exp(tH_{p_f})\quad 
\text{so that }w_{c,\pm} =\lim_{t\to\pm\infty} w_c(t).
\]
Let $U\subset \re_+\times\re\times (T^*\pa M\setminus 0)$ be a relatively compact domain. 
Then the convergence of $w_c(t)^{-1}$ to $w_{c,\pm}^{-1}$ (as $t\to\pm\infty$) is 
uniform on $U$ with all the derivatives. Since the limit is diffeomorphic, its 
inverse $w(t)$ also has the same property (on $w_c(t)^{-1}U$). 
In particular, all the derivatives of $w_c(t)^{-1}$ on $U$ are uniformly bounded in $t$, 
and all the derivatives of $w_c(t)$ on $w_c(t)^{-1}U$ are uniformly bounded. 

We note that it is easy to check $w_{c,\pm}$ and hence $s_c$ are homogeneous of 
order one with respect to $(r,\o)$-variables, i.e.,
\[
w_{c,\pm}^{-1}(\l r_0, \rho_0,\th_0,\l \o_0)= (\l r_\pm, \rho_\pm, \th_\pm, \l \o)
\quad \text{for }\l>0.
\]
This is consistent with the scaling property of $w_c(t)$: 
\[
w_c^{-1}(\l t)(\l r_0,\rho_0,\th_0,\l \o_0) =(\l\tilde r(t),\rho(t),\th(t),\l\o(t))
\]
for any $\l>0, t\in\re$. 

\subsection{Classical flow generated by the scattering metric}

Here we discuss the Hamilton flow generated by the symbol of $P$:
\[
p(r,\rho,\th,\o)= \frac12\biggpare{a_1(r,\th)\rho^2 +\frac{2\rho a_2(r,\th)\cdot \o}{r} 
+\frac{\o\cdot a_3(r,\th)\o}{r^2}}+V
\]
on $T^*M_\infty$. 

We let $\O_0\Subset T^*\re_+\times (T^*\pa M\setminus 0)$.
For $h\in (0,1]$, we set
\[
\O_0^h =\bigset{(r,\rho,\th,\o)\in T^*\re_+\times (T^*\pa M\setminus 0)}{(hr,\rho,\th, h\o)
\in \O_0},
\]
and we consider the Hamilton flow with initial conditions in $\O_0^h$. 
We show that if $h$ is sufficiently small then the classical (inverse) wave operators exist on $\O_0^h$, 
and they are very close to $w_{c,\pm}^{-1}$, the (inverse) wave operators for the conic metric. 

\begin{thm}
\begin{enumerate}
\renewcommand{\theenumi}{\roman{enumi}}
\renewcommand{\labelenumi}{{{\rm (\theenumi)} }}
\item Let $\O_0$ and $\O_0^h$ as above. Then there is $h_0>0$ such that if $h\in (0,h_0]$
\[
w^*_\pm(r,\rho,\th,\o) :=\lim_{t\to\pm\infty}
\exp(-tH_{p_f})\circ \exp(tH_p)(r,\rho,\th,\o)
\]
exist for $(r,\rho,\th,\o)\in \O_0^h$, and the convergence holds in the 
$C^\infty$-topology on $\O_0^h$. 
\item We denote 
\begin{align*}
&(r(t),\rho(t),\th(t),\o(t))= \exp(tH_p)(r,\rho,\th,\o), \\
&(r_c(t),\rho_c(t),\th_c(t),\o_c(t))=\exp(tH_{p_c})(r,\rho,\th,\o)
\end{align*}
for $(r,\rho,\th,\o)\in \O_0^h$. Then for any indices $\a,\b,\c$ and $\d$, there is $C>0$ 
such that 
\begin{align*}
&\bigabs{\pa_r^\a\pa_\rho^\b\pa_\th^\c\pa_\o^\d (r(t)-r_c(t))}
+\bigabs{\pa_r^\a\pa_\rho^\b\pa_\th^\c\pa_\o^\d (\o(t)-\o_c(t))}
\leq C h^{-1+\m+|\a|+|\d|}, \\
&\bigabs{\pa_r^\a\pa_\rho^\b\pa_\th^\c\pa_\o^\d (\rho(t)-\rho_c(t))}
+\bigabs{\pa_r^\a\pa_\rho^\b\pa_\th^\c\pa_\o^\d (\th(t)-\th_c(t))}
\leq C h^{\m+|\a|+|\d|}, 
\end{align*}
for $(r,\rho,\th,\o)\in \O_0^h$, $t\in\re$, $0<h\leq h_0$. 
\item If we denote
\[
w^*_\pm(r,\rho,\th,\o)=(r_\pm,\rho_\pm,\th_\pm, \o_\pm), \quad 
w^*_{c,\pm}(r,\rho,\th,\o)=(r_{c,\pm},\rho_{c,\pm},\th_{c,\pm}, \o_{c,\pm}),
\]
then 
\begin{align*}
&\bigabs{\pa_r^\a\pa_\rho^\b\pa_\th^\c\pa_\o^\d (r_\pm-r_{c,\pm})}
+\bigabs{\pa_r^\a\pa_\rho^\b\pa_\th^\c\pa_\o^\d (\o_\pm-\o_{c,\pm})}
\leq C h^{-1+\m+|\a|+|\d|}, \\
& \bigabs{\pa_r^\a\pa_\rho^\b\pa_\th^\c\pa_\o^\d (\rho_\pm-\rho_{c,\pm})}
+\bigabs{\pa_r^\a\pa_\rho^\b\pa_\th^\c\pa_\o^\d (\th_\pm-\th_{c,\pm})}
\leq C h^{\m+|\a|+|\d|}
\end{align*}
for $(r,\rho,\th,\o)\in \O_0^h$, $0<h\leq h_0$. 
\end{enumerate}
\end{thm}

For $(r_0,\rho_0,\th_0,\o_0)\in \O_0$, we set $(r^h(t),\rho^h(t),\th^h(t),\o^h(t))$ so that 
\[
(h^{-1}r^h(t),\rho^h(t),\th^h(t),h^{-1}\o^h(t))=\exp(h^{-1} tH_p)(h^{-1}r_0,\rho_0,\th_0,h^{-1}\o_0).
\]
We also set 
\[
p^h(r,\rho,\th,\o) = p(h^{-1}r,\rho,\th,h^{-1}\o), \quad (r,\rho,\th,\o)\in T^*M_\infty.
\]
Then it is easy to check 
\[
(r^h(t),\rho^h(t),\th^h(t),\o^h(t))= \exp(tH_{p^h})(r_0,\rho_0,\th_0,\o_0).
\]
On the other hand, if we write 
\[
p^h(r,\rho,\th,\o) =p_c(r,\rho,\th,\o)+v^h(r,\rho,\th,\o), 
\]
then we learn by the Assumption~A that for any indices $\a,\b,\c,\d$, 
\begin{multline}\label{est:vh}
\bigabs{\pa_r^\a\pa_\rho^\b\pa_\th^\c\pa_\o^\d v^h(r,\rho,\th,\o)}\\
\leq C_{\a\b\c\d} h^\m\bigpare{r^{-1}\jap{\rho}^2+r^{-1}\jap{\rho}\jap{\o}+r^{-2}\jap{\o}^2} r^{-\m-|\a|}
\jap{\rho}^{-|\b|} \jap{\o}^{-|\d|}.
\end{multline}

In order to prove Theorem~2.1, it suffices to show: 

\begin{thm}
\begin{enumerate}
\renewcommand{\theenumi}{\roman{enumi}}
\renewcommand{\labelenumi}{{{\rm (\theenumi)} }}
\item There is $h_0>0$ such that if $h\in (0,h_0]$ then 
\[
w^*_{\pm,h}(r_0,\rho_0,\th_0,\o_0) :=\lim_{t\to\pm\infty} \exp(-tH_{p_f})\circ\exp(tH_{p^h})
(r_0,\rho_0,\th_0,\o_0)
\]
exist for $(r_0,\rho_0,\th_0,\o_0)\in \O_0$, and the convergence holds in the $C^\infty$-topology. 
\item For any indices $\a,\b,\c,\d$, there is $C>0$ such that 
\begin{align*}
&\bigabs{\pa_{r_0}^\a\pa_{\rho_0}^\b\pa_{\th_0}^\c\pa_{\o_0}^\d (r^h(t)-r_c(t))}
+\bigabs{\pa_{r_0}^\a\pa_{\rho_0}^\b\pa_{\th_0}^\c\pa_{\o_0}^\d (\rho^h(t)-\rho_c(t))} \\
&+\bigabs{\pa_{r_0}^\a\pa_{\rho_0}^\b\pa_{\th_0}^\c\pa_{\o_0}^\d (\th^h(t)-\th_c(t))}
+\bigabs{\pa_{r_0}^\a\pa_{\rho_0}^\b\pa_{\th_0}^\c\pa_{\o_0}^\d (\o^h(t)-\o_c(t))}
\leq Ch^\m
\end{align*}
for $(r_0,\rho_0,\th_0,\o_0)\in \O_0$, $h\in (0,h_0]$, $t\in\re$, where 
\[
(r_c(t),\rho_c(t),\th_c(t),\o_c(t)) =\exp(tH_{p_c})(r_0,\rho_0,\th_0,\o_0).
\]
\item Denoting 
\[
(r_\pm^h,\rho_\pm^h,\th_\pm^h,\o_\pm^h) =w_{\pm,h}^*(r_0,\rho_0,\th_0,\o_0),
\]
we have for any indices $\a,\b,\c,\d$, 
\begin{align*}
&\bigabs{\pa_{r_0}^\a\pa_{\rho_0}^\b\pa_{\th_0}^\c\pa_{\o_0}^\d (r^h_\pm-r_{c,\pm})}
+\bigabs{\pa_{r_0}^\a\pa_{\rho_0}^\b\pa_{\th_0}^\c\pa_{\o_0}^\d(\rho^h_\pm-\rho_{c,\pm})} \\
&+\bigabs{\pa_{r_0}^\a\pa_{\rho_0}^\b\pa_{\th_0}^\c\pa_{\o_0}^\d (\th^h_\pm-\th_{c,\pm})}
+\bigabs{\pa_{r_0}^\a\pa_{\rho_0}^\b\pa_{\th_0}^\c\pa_{\o_0}^\d (\o^h_\pm-\o_{c,\pm})}
\leq Ch^\m
\end{align*}
\end{enumerate}
\end{thm}

\begin{proof}
The proof is analogous to the argument in \cite{Na2} Section~2, and \cite{IN1} Section~2. 
We only outline the proof, and we omit the detail. 

\noindent{\bf Step 1.} By the standard virial-type argument, we learn that 
there is $R>0$ such that 
\[
\frac{d^2}{dt^2} (r^h(t)^2) \geq c>0 \quad\text{if }r^h(t)\geq R,
\]
if $(r_0,\rho_0,\th_0,\o_0)\in\O_0$. 
Here we use the fact that $|\rho|$ and $|\o/r|$ are uniformly bounded by the 
conservation of energy. On the other hand, since $v^h = O(h^\m)$, we also 
learn that $r^h(t)\to r_c(t)$ as $h\downarrow 0$, locally uniformly in $t$. 
Thus, if $t_0$ is large and $h$ is small enough, $r^h(t)\geq R$, and 
combining this with the above observation, we have 
\[
|r^h(t)|\geq \sqrt{R+c|t-t_0|^2/2} \quad \text{for }t\geq t_0. 
\]
Hence we learn 
\[
c_1\jap{t}\leq r^h(t) \leq c_2\jap{t}\quad \text{for }h\in (0,h_0], t>0
\]
with some $h_0, c_1,c_2>0$. The case  $t<0$ can be considered similarly. 

\noindent{\bf Step 2.} We consider the time evolution of $q_0(t)=q(\th^h(t),\o^h(t))$. 
By the Hamilton equation and \eqref{est:vh}, we have 
\begin{align*}
\frac{d}{dt} q_0(t) &= -\{p^h,q_0\} =-\{v^h,q_0\} \\
&= O(h^\m r^{-1-\m}\jap{\o}^2) = O(h^\m \jap{t}^{-1-\m}(1+q_0(t))).
\end{align*}
Here we have used the boundedness of $|\rho(t)|$ and $|\o(t)/r(t)|$ again. 
Then by the Duhamel formula, we learn that $q_0(t)$ is uniformly bounded 
for the initial conditions in $\O_0$ and $h\in (0,h_0]$. This implies 
$|\o^h(t)|$ is also uniformly bounded. 

\noindent{\bf Step 3.} Combining the above observations with the Hamilton equation, 
we learn 
\begin{align*}
&\biggabs{\frac{d}{dt} \rho^h(t)}\leq C\jap{t}^{-2-\m}, \quad 
\biggabs{\frac{d}{dt}(r^h(t)-t\rho^h(t))}\leq C\jap{t}^{-1-\m}, \\
&\biggabs{\frac{d}{dt} \th^h(t)}\leq C\jap{t}^{-1-\m}, \quad 
\biggabs{\frac{d}{dt} \o^h(t)}\leq C\jap{t}^{-1-\m}, 
\end{align*}
uniformly for $(r_0,\rho_0,\th_0,\o_0)\in\O_0$, $h\in (0,h_0]$ and $t\in\re$. 
These imply the existence of $w^*_{\pm,h}$ on $\O_0$. 
We can show the similar estimates for the derivatives, i.e., 
\[
\biggabs{\frac{d}{dt} \bigpare{\pa_{r_0}^\a \pa_{\rho_0}^\b \pa_{\th_0}^\c \pa_{\o_0}^\d 
\rho^h(t)}}\leq C\jap{t}^{-2-\m-|\a|}, 
\quad \text{etc.}
\]
These imply the convergence in $C^\infty$-topology, and we conclude the assertion (i). 

\noindent{\bf Step 4.} We set
\[
g^h(t)=|r^h(t)-r_c(t)|+|\rho^h(t)-\rho_c(t)|+|\th^h(t)-\th_c(t)|+|\o^h(t)-\o_c(t)|.
\]
Then by the Hamilton equation, \eqref{est:vh}, and the estimates in Steps 1 and 2, we learn 
\[
\biggabs{\frac{d}{dt} g^h(t)} \leq C \jap{t}^{-1-\m}g^h(t) + Ch^\m\jap{t}^{-1-\m},
\]
uniformly for the initial condition in $\O_0$ and $h\in (0,h_0]$. 
Then by the Duhamel formula and noting $g^h(0)=0$, we obtain 
\[
\bigabs{g^h(t)}\leq C h^\m, \quad t\in\re.
\]
This is the assertion (ii) with $\a=\b=\c=\d=0$. The derivatives can be estimated similarly 
by induction. For the detail of this argument, we refer \cite{CKS} Section~2, or \cite{Na2}
Section~2. 
The assertion (iii) follows immediately from the assertion (ii). 
\end{proof}

By the above argument, we also learn $w^*_{\pm,h}$ are invertible for small $h$. 
The inverses are uniformly bounded, and their inverses 
\[
w_{\pm,h}=(w_{\pm,h}^*)^{-1}
\]
are well-defined for $h\in (0,h_0]$. It follows that 
\[
w_{\pm}= (w_{\pm}^*)^{-1}
\]
is well-defined  and diffeomorphic on $w_{\pm}^*[\O_0^h]$ with $h\in (0,h_0]$. 
Thus we can define the classical scattering operator by 
\[
s= w_+^*\circ w_-
\]
on $w^*_-[\O_0^h]$ with sufficiently small $h$.


\section{Symbol classes and their quantization on scattering manifolds}

Here we prepare a pseudodifferential operator calculus which is used extensively in 
the proof of main theorems. We refer to textbooks by H\"ormander \cite{Ho} and Taylor \cite{Ta} for 
the standard theory of microlocal analysis. 

In the following, we employ symbol calculus on $T^*M$, but we always suppose the 
symbol is supported in $T^*M_\infty$, and we use the local coordinate system as in Section~1. 
More specifically, we choose a local coordinate system on $\pa M$: $\{\f_\a: U_\a\to \re^{n-1}\}$, 
$U_\a\subset \pa M$, and we use the coordinate system 
$\{1\otimes\f_\a: \re_+\times U_\a\to \re\times \re^{n-1}\}$ on $M_\infty$. 
We also use a similar local coordinate system on $M_f$ defined by 
$\{1\otimes\f_\a: \re\times U_\a\to \re\times \re^{n-1}\}$. 
We often identify $U_\a$ (or $\re_+\times U_\a$, $\re\times U_\a$, resp.) with $\Ran\f_\a$
(or $\Ran (1\otimes \f_\a)$, resp.). 

\subsection{Symbol classes}

We set a metric either on $T^*M_\infty$ or $T^*M_f$ defined by
\[
g_1 = \frac{dr^2}{\jap{r}^2} + d\rho^2 + d\th^2 +\frac{d\o^2}{\jap{\o}^2},
\]
and consider symbols in $S(m,g_1)$ with a weight function $m$, i.e., 
$a\in S(m,g_1)$ if and only if for any indices $a,\b,\c,\d$, there is $C$ such that 
\[
\bigabs{\pa_r^\a\pa_\rho^\b\pa_\th^\c\pa_\o^\d a(r,\rho,\th,\o)} \leq 
C m(r,\rho,\th,\o)\jap{r}^{-|\a|}\jap{\o}^{-|\d|}. 
\]
Later, we will consider the symbol calculus of such symbols on sets of the form:
$\O^h=\{(r,\rho,\th,\o)\,|\, (hr,\rho,\th,h\o)\in\O\}$ with some compact set 
$\O\subset T^*M_\infty$ (supported away from $\{\o=0\}$) and small $h>0$. 
In such cases, the symbol satisfies 
\[
\bigabs{\pa_r^\a\pa_\rho^\b\pa_\th^\c\pa_\o^\d a(h;r,\rho,\th,\o)} \leq 
C m(h) h^{|\a|+|\d|},
\]
and we denote such ($h$-dependent) symbol as $a\in S_h(m,g_1^h)$, where 
$m$ is an $h$-dependent weight. The corresponding metric is naturally 
\[
g_1^h =h^2 dr^2 +d\rho^2+d\th^2+h^2d\o^2. 
\]

\subsection{Weyl quantization}
Let $\{\chi^2_\a\}$ be a partition of unity on $\pa M$ compatible with our coordinate 
system $\{\f_\a, U_\a\}$, i.e., $\chi_\a\in C_0^\infty(U_\a)$ and 
$\sum_{\a} \chi_\a(\th)^2 \equiv 1$ on $\pa M$. We also denote 
$\tilde \chi_\a(r,\th)= \chi_\a(\th) j(r)\in C^\infty(M_\infty)$. 
Let $a\in S(m,g_1)$ be a symbol on $T^*M_\infty$, and let
$u\in C_0^\infty(T^*M)$. We denote by $a_{(\a)}$ and $G_{(\a)}$ 
the representation of $a$ and $G$ in the local coordinate 
$(1\otimes \f_\a,\re\times U_\a)$, respectively. 
We quantize $a$ by 
\[
Op^W(a)u = \sum_\a \tilde \chi_\a G_{(\a)}^{-1/2} a^W_{(\a)}(r, D_r, \th, D_\th) G_{(\a)}^{1/2}
\tilde \chi_\a u
\]
where $a^W_{(\a)}(r,D_r,\th,D_\th)$ denotes the usual Weyl-quantization on the 
Euclidean space $\re^n$, and we use the identification: 
$\re_+\times U_\a\cong \re_+\times (\Ran\f_\a)$ for each $\a$. 
(Strictly speaking, we should have written it as 
\[
Op^W(a)u = \sum_\a \tilde \chi_\a (\tilde\f_\a)^*\bigpare{G_{(\a)}^{-1/2} a^W_{(\a)}(r, D_r, \th, D_\th) G_{(\a)}^{1/2}
(\tilde \f_\a)_*(\tilde \chi_\a u)},
\]
but we will omit $(\tilde\f_\a)_*$, $(\tilde\f_\a)^*$, etc., without confusions.)
This definition is compatible with the standard definition of pseudodifferential operators 
on manifolds, but we choose specific quantization which preserves the 
asymptotically conic structure of $M$. 
Similarly, for a symbol $a$ on $T^*M_f$, we quantize it by 
\[
Op^W(a)u = \sum_\a \chi_\a H_{(\a)}^{-1/2} a^W_{(\a)}(r, D_r, \th, D_\th) H_{(\a)}^{1/2}
\chi_\a u
\]
for $u\in C_0^\infty(M_f)$, where $H_{(\a)}$ denotes the representation of $H$ in the local coordinate 
$(\f_\a,U_\a)$. In this case, the linear structure in $r$ is preserved. 

In the above definition, we put weights around the locally defined pseudodifferential operators $a_{\a}^W$ 
so that $Op^W(a)$ is symmetric if $a$ is real-valued. Moreover, by virtue of these weights, 
the symbol corresponding to the operator is unique 
including the subprincipal symbol, though we will not take advantage of this fact in this paper. 

The above definition of quantizations also have a convenient property that if we identify 
a symbol $a$ on $T^*M_\infty$ with a symbol on $T^*M_f$ (by the obvious identification), 
then we have 
\[
\mathcal{I} Op^W(a) \mathcal{I}^* = Op^W(a)\quad \text{on } \mathcal{H}
\]
provided $a$ is supported in $\{r>1\}$, and we may identify these quantizations 
by using $\mathcal{I}$. For a symbol supported in $\{r>1\}$, we may consider 
$Op^W(a)$ as an operator from $\mathcal{H}$ to $\mathcal{H}_f$ 
(or from $\mathcal{H}_f$ to $\mathcal{H}$) also. We define them by 
\[
Op^W(a)u = \sum_\a \chi_\a H_{(\a)}^{-1/2} a^W_{(\a)}(r, D_r, \th, D_\th) G_{(\a)}^{1/2}
\tilde \chi_\a u
\]
for $u\in C_0^\infty(M)$ and 
\[
Op^W(a)u = \sum_\a \tilde \chi_\a G_{(\a)}^{-1/2} a^W_{(\a)}(r, D_r, \th, D_\th) H_{(\a)}^{1/2}
\chi_\a u
\]
for $u\in C_0^\infty(M_f)$, respectively. 

If $A=Op^W(a)$, then we denote the (Weyl) symbol of $A$ by $a=\S(A)$. 

\subsection{Hamiltonians}

Now we consider properties of our Schr\"odinger operators and related operators 
as a preparation of the next section. 

We note that, as in the usual Weyl calculus on $\re^n$,  if $a(x,\x)= \sum_{j,k} a_{jk}(x)\x_j\x_k$, then 
\[
Op^W(a)=\sum_{j,k} D_j a_{jk}(x) D_k -\frac14 \sum_{j,k} (\pa_j\pa_k a_{jk}(x)).
\]
Hence, if we let $p$ be the symbol of $P$ as in Subsection~2.3, then we have 
\[
Op^W(p)= P+f,
\]
where $f\in C^\infty(M_f)$ such that 
\[
\bigabs{\pa_r^\a\pa_\th^\b f(r,\th)}\leq C_{\a\b}\jap{r}^{-2-|\a|}
\]
for any $\a,\b$. Thus, we can include this error term in $V$ and we may consider 
$P=Op^W(p)$. On the other hand, it is easy to see $P_f=Op^W(p_f)$ 
on $\mathcal{H}_f$ where $p_f=\frac12 \rho^2$.


\section{Egorov theorem}

Let $(r_0,\rho_0,\th_0,\o_0)\in T^*(\re_+\times \pa M)$, $\o_0\neq 0$, and suppose 
$a\in C_0^\infty(T^*(\re_+\times \pa M))$ is supported in a small neighborhood of 
$(r_0,\rho_0,\th_0,\o_0)$ so that $a$ is supported away from $\{\o=0\}$. 
We set
\[
a^h(r,\rho,\th,\o)= a(h;hr,\rho,\th,h\o), \quad h>0,
\]
where $a$ itself may depend on the parameter $h>0$, but we suppose 
it is bounded uniformly in $C_0^\infty$-topology, and supported in the 
same small neighborhood of $(r_0,\rho_0,\th_0,\o_0)$. We note the notation here is 
different from that of Section~2. Then we set
\[
A_0=Op^W(a^h) \quad \text{on }M. 
\]
We set $\e>0$ so small that 
\[
\exp(tH_{p_c})(\supp a)\cap \{r\leq \e \jap{t}\} =  \emptyset 
\]
for all $t\in\re$. 
We choose $\y\in C^\infty(\re)$ such that $\y(r)=1$ for $r\geq 1$ and $\y(r)=0$ for 
$r\leq 1/2$, and we set 
\[
Y=\y\bigpare{\tfrac{hr}{\e\jap{t}}}.
\]
Then we define 
\[
A(t) =e^{itP_f/h} \mathcal{I}^* Y e^{-itP/h} A_0 
e^{itP/h} Y \mathcal{I} e^{-itP_f/h}
\]
for $t\in\re$. The purpose of this section is to obtain the symbols of $A(t)$ as a pseudodifferential 
operator, and to study its behavior as $t\to\pm\infty$. 

We compute (formally):
\[
\frac{d}{dt}\Bigpare{e^{itP/h} Y\I e^{-itP_f/h}} 
=\frac{i}{h} e^{itP/h} T(t) e^{-itP_f/h},
\]
where 
\[
T(t)=PY\I -Y\I P_f -\frac{h(hr)t}{i\e\jap{t}^3} \y'\bigpare{\tfrac{hr}{\e\jap{t}}} \I.
\]
We further rewrite this as 
\begin{align*}
\frac{d}{dt}\Bigpare{e^{itP/h} Y\I e^{-itP_f/h}}  
&= \frac{i}{h} \Bigpare{e^{itP/h} Y\I e^{-itP_f/h}} 
\Bigpare{e^{itP_f/h} \I^* T(t) e^{-itP_f/h}} \\
&\quad + \frac{i}{h} e^{itP/h}(1-Y\I\I^*) T(t) e^{-itP_f/h} \\
&= \frac{i}{h} \Bigpare{e^{itP/h} Y\I e^{-itP_f/h}} L(t) +R_1(t)
\end{align*}
where 
\[
L(t)= e^{itP_f/h} \I^* T(t) e^{-itP_f/h}, \quad 
R_1(t)= \frac{i}{h} e^{itP/h} (1-Y\I\I^*) T(t) e^{-itP_f/h}.
\]
We now consider the symbols of $T(t)$ and $L(t)$ as pseudodifferential operators. 
By direct computations, it is easy to see that for any indices $\a,\b,\c,\d$, 
\begin{align}\label{eq:es1}
&\bigabs{\pa_r^\a\pa_\rho^\b\pa_\th^\c\pa_\o^\d \S(T(t))(r,\rho,\th,\o)}\\
&\quad \leq C\Bigpare{\jap{r}^{-1-\m}\jap{\rho}^2 +\jap{r}^{-1-\m}\jap{\rho}\jap{\o} 
+\jap{r}^{-2}\jap{\o}^2} \jap{r}^{-|\a|} \jap{\rho}^{-|\b|} \jap{\o}^{-|\d|}.\nonumber 
\end{align}
Since $T(t)$ is supported in $\{r\geq \e\jap{t}/2h\}$, we may replace 
$\jap{r}$ by $\jap{r}+\e\jap{t}/2h$ in the above estimate. We also have 
\begin{align*}
&\biggabs{\pa_r^\a\pa_\rho^\b\pa_\th^\c\pa_\o^\d \biggpare{\S(T(t))-
Y \frac{q(\th,\o)}{r^2}}}\\
&\quad \leq C\Bigpare{\jap{r}^{-1-\m}\jap{\rho}^2 +\jap{r}^{-1-\m}\jap{\rho}\jap{\o} 
+\jap{r}^{-2-\m}\jap{\o}^2} \jap{r}^{-|\a|} \jap{\rho}^{-|\b|} \jap{\o}^{-|\d|}.\nonumber 
\end{align*}
In particular, we learn 
\begin{equation}\label{eq:es2}
\biggabs{\pa_r^\a\pa_\rho^\b\pa_\th^\c\pa_\o^\d \biggpare{\S(T(t))-
Y \frac{q(\th,\o)}{r^2}}} 
\leq C\jap{t}^{-1-\m-|\a|} h^{\m +|\a|+|\d|}
\end{equation}
on $\exp(tH_{p_c})\bigbrac{\supp a^h}$, where the constant is independent of 
$t$ and $h$. 

Now we note, by virtue of the Weyl-calculus (and our choice of the quantization), 
\[
\S(L(t))(r,\rho,\th,\o)= \S(\I^*T(t))(r+(t/h)\rho,\rho,\th,\o). 
\]
Hence we have, by \eqref{eq:es1}, 
\begin{align*}
&\bigabs{\pa_r^\a\pa_\rho^\b\pa_\th^\c\pa_\o^\d \S(L(t))}\\
&\quad \leq C \Bigpare{\jap{\tilde r}^{-1-\m} \jap{\rho}^2+\jap{\tilde r}^{-1-\m} \jap{\rho}\jap{\o} 
+\jap{\tilde r}^{-2}\jap{\o}^2 }\jap{\tilde r}^{-|\a|}\jap{\o}^{-|\d|},
\end{align*}
where $\tilde r =r+(t/h)\rho$. We note we take advantage of the cut-off function $Y$ in this 
estimate. We also note, as well as \eqref{eq:es2}, 
\begin{equation}\label{eq:es3}
\biggabs{\pa_r^\a\pa_\rho^\b\pa_\th^\c\pa_\o^\d \biggpare{\S(L(t))-
\frac{q(\th,\o)}{\tilde r^2}}} 
\leq C\jap{t}^{-1-\m-|\a|} h^{\m+|\a|+|\d|}
\end{equation}
on $\exp(-tH_{p_f})\circ \exp(tH_{p_c})\bigbrac{\supp a^h} = \supp(a^h\circ w_c(t))$. 

We then construct an asymptotic solution to the Heisenberg equation: 
\begin{equation}\label{eq:Heiseiberg}
\frac{d}{dt} B(t) =-\frac{i}{h} [L(t), B(t)], \quad B(0)= \I^*  A_0 \I . 
\end{equation}

\begin{lem}
There exists $b^h(t;r,\rho,\th,\o)\in C_0^\infty(T^*M_f)$ such that 
\begin{enumerate}
\renewcommand{\theenumi}{\roman{enumi}}
\renewcommand{\labelenumi}{{{\rm (\theenumi)} }}
\item $b^h(0)=a^h$. 
\item $b^h(t)$ is supported in $w_c(t/h)^{-1} \bigbrac{\supp a^h}$. 
\item $b^h(t)\in S(1,g_1^h)$, and it is bounded uniformly in $t\in\re$. 
\item $b^h(t)-a^h\circ w_c(t/h)\in S(h^\m,g_1^h)$, i.e., the principal symbol of $b^h(t)$ 
is given by $a^h\circ w_c(t/h)$, and the remainder is bounded uniformly in $t$. 
\item If we set $B(t)=Op^W(b^h(t))$, then 
\[
\Bignorm{\frac{d}{dt}B(t)+\frac{i}{h}[L(t),B(t)]} \leq C_N \jap{t}^{-1-\m} h^N, \quad h>0, 
\]
with any $N$.
\item $B(t)$ converges to $B_\pm$ as $t\to\pm\infty$ in $B(\mathcal{H}_f)$, and 
the symbols : $b^h_\pm :=\S(B_\pm)$ satisfy 
\[
b^h_\pm-a^h\circ w_{c,\pm} \in S(h^\m,g_1^h).
\]
\end{enumerate}
\end{lem}

\begin{proof}
We follow the standard procedure to construct asymptotic solutions to 
Heisenberg equations (see, e.g., \cite{Ta} Chapter 8, \cite{Ma} Chapter 4). 
We let 
\[
\ell_0(t;r,\rho,\th,\o) = \frac{q(\th,\o)}
{(r+t\rho)^2}
\]
be the principal symbol of $L(ht)$. If we set 
\[
b_0(t)=a\circ w_c(t) = a \circ \exp(-tH_{p_c})\circ \exp(tH_{p_f}),
\]
then $b_0$ satisfies the equation:
\[
\frac{\pa}{\pa t} b_0(t) =-\{\ell_0(t),b_0(t)\}, \quad b_0(0)= a.
\]
We set $b_0^h(t;r,\rho,\th,\o)= b_0(t/h;hr,\rho,\th,h\o)$, 
and we also set $B_0(t)=Op^W(b_0^h(t))$. We note 
\[
\bigabs{\pa_r^\a\pa_\rho^\b\pa_\th^\c\pa_\o^\d b^h_0(t;r,\rho,\th,\o)}\leq C h^{|\a|+|\d|}
\]
uniformly in $t$ with any $\a,\b,\c,\d$, since $b_0(t)$ converges to $a\circ w_{c,\pm}$ as
$t\to\pm \infty$. 
We write
\[
R^0_0(t) = \frac{d}{dt} B_0(t)+\frac{i}{h}[L(t),B_0(t)], \quad r^0_0(t)=\S(R^0_0(t)).
\]
Then by \eqref{eq:es3} and the symbol calculus, $r^0_0(t)$ is supported on 
$w_c(t/h)^{-1}\bigbrac{\supp a^h}$ modulo $O(h^\infty)$-terms, and 
\begin{equation}\label{eq:es4}
\bigabs{\pa_r^\a\pa_\rho^\b\pa_\th^\c\pa_\o^\d r^0_0(t)}\leq C\jap{t}^{-1-\m-|\a|}
h^{\m+|\a|+|\d|}
\end{equation}
with any $\a,\b,\c,\d$. We set $\tilde r_0^0(t)$ so that 
\[
\tilde r^0_0(t/h;hr,\rho,\th,h\o)= r^0_0(t;r,\rho,\th,\o), 
\]
and solve the transport equation:
\[
\frac{\pa}{\pa t} b_1(t) +\{\ell_0(t),b_1(t)\} =-\tilde r^0_0(t), \quad b_1(0)=0.
\]
By \eqref{eq:es4}, it is easy to observe
\[
\bigabs{\pa_r^\a\pa_\rho^\b\pa_\th^\c\pa_\o^\d b_1(t;r,\rho,\th,\o)}
\leq C h^\m
\]
uniformly in $t$. 
Moreover, $b_1(t)$ converges to a symbol supported in $w_{c,\pm}^{-1}\bigbrac{\supp a}$ 
in $C_0^\infty$-topology as $t\to\pm\infty$. 
We then set 
\[
B_1(t) =Op^W(b_1^h(t)), \quad b_1^h(t;r,\rho,\th,\o)=b_1(t/h;hr,\rho,\th,h\o). 
\]
We construct $b_j$, $j=1,2,\dots$, iteratively, so that $b_j^h \in S(h^{j\m},g_1^h)$, and set 
\[
b^h(t)\sim \sum_{j=0}^\infty b_j^h(t), \quad B(t)=Op^W(b^h(t)).
\]
By the construction, $b^h(t)$ and $B(t)=Op^W(b^h(t))$ satisfy the assertion. 
\end{proof}

We then observe that $A(t)$ is very close to $B(t)$ constructed as above. 

\begin{lem} For any $N$, there is $C_N>0$ such that 
\[
\norm{A(t)-B(t)}\leq C_N h^N, \quad t\in\re.
\]
In particular, 
\[
A_\pm := \wlim_{t\to\pm\infty} A(t)
\]
have the symbols $b^h_\pm$
as pseudodifferential operators. 
\end{lem}

\begin{proof}
We first observe 
\begin{align*}
\norm{A(t)-B(t)} &= \bignorm{e^{itP_f/h} \I^* Y e^{-itP/h} A_0 e^{itP/h} Y \I e^{-itP_f/h} -B(t)}\\
&=\bignorm{\I^* Y e^{-itP/h} A_0 e^{itP/h} Y \I -e^{-itP_f/h}B(t)e^{itP_f/h}}\\
&\leq \bignorm{Y\I\I^*Y e^{-itP/h} A_0 e^{itP/h} Y \I\I^* Y -Y\I e^{-itP_f/h}B(t)e^{itP_f/h}\I^* Y}\\
&\leq \bignorm{e^{-itP/h} A_0 e^{itP/h}  -Y\I e^{-itP_f/h}B(t)e^{itP_f/h}\I^* Y}
+R_2\\
&= \bignorm{A_0- \tilde B(t)}+R_2,
\end{align*}
where 
\[
R_2 = 2\bignorm{(1-Y\I\I^*Y)e^{-itP/h} A_0}
\]
and 
\[
\tilde B(t) = e^{itP/h} Y\I e^{-itP_f/h}B(t)e^{itP_f/h}\I^* Y e^{-itP/h}. 
\]
By Corollary A.2 in Appendix A, we learn $R_2= O(\jap{t}^{-N}h^N)$ with any $N$.
We then show $\tilde B(t)$ is very close to $A_0$ uniformly in $t$. 
We compute
\begin{align*}
\frac{d}{dt}\tilde B(t) 
&= \bigpare{e^{itP/h} Y\I e^{-itP_f/h}}\frac{d}{dt}B(t)\bigpare{e^{itP_f/h}\I^* Y e^{-itP/h}}\\
&\quad + \frac{i}{h}\bigpare{ e^{itP/h} Y\I e^{-itP_f/h}} L(t) B(t) \bigpare{e^{itP_f/h} \I^* Y e^{-itP/h}}\\
& \quad - \frac{i}{h}\bigpare{ e^{itP/h} Y\I e^{-itP_f/h}} B(t) L(t)^* \bigpare{e^{itP_f/h} \I^* Y e^{-itP/h}}\\
&\quad +R_1(t) B(t) \bigpare{e^{itP_f/h} \I^* Y e^{-itP/h}} 
-\bigpare{e^{itP/h} Y\I e^{-itP_f/h}} B(t) R_1(t)^*\\
&= \bigpare{e^{itP/h} Y\I e^{-itP_f/h}}\biggpare{\frac{d}{dt}B(t)+\frac{i}{h}[L(t),B(t)]}
\bigpare{e^{itP_f/h}\I^* Y e^{-itP/h}} \\
&\quad +R_3(t)
\end{align*}
where 
\begin{align*}
R_3(t) &= R_1(t) B(t) \bigpare{e^{itP_f/h} \I^* Y e^{-itP/h}} 
-\bigpare{e^{itP/h} Y\I e^{-itP_f/h}} B(t) R_1(t)^*\\
&\quad 
+\frac{i}{h}\bigpare{ e^{itP/h} Y\I e^{-itP_f/h}} B(t) (L(t)-L(t)^*) \bigpare{e^{itP_f/h} \I^* Y e^{-itP/h}}.
\end{align*}
We can show $\norm{R_3(t)}=O(\jap{t}^{-N} h^N)$ with any $N$. For example, 
\begin{align*}
&\bignorm{R_1(t) B(t) \bigpare{e^{itP_f/h} \I^* Y e^{-itP/h}}} 
\leq h^{-1} \bignorm{(1-Y\I\I^*)T(t) e^{-itP_f/h} B(t)}\\
&\qquad = h^{-1} \bignorm{e^{itP_f/h}(1-Y\I\I^*)T(t) e^{-itP_f/h} B(t)}.
\end{align*}
As we have seen already, $e^{itP_f/h}(1-Y\I\I^*)T(t) e^{-itP_f/h}$ is a pseudodifferential 
operator, and the support is separated from the support of $b^h(t)$. Moreover,  
the distance is bounded from below by $c\jap{t}h^{-1}$, $c>0$. Thus the product has 
a vanishing symbol, and the norm is $O(\jap{t}^{-N}h^N)$ with any $N$. The other terms 
are estimated similarly. 
Combining this with Lemma~4.1-(v), we learn 
\[
\bignorm{\tfrac{d}{dt} \tilde B(t)} \leq C_N \jap{t}^{-1-\m} h^N
\]
with any $N$, and hence 
$\bignorm{\tilde B(t)-\tilde B(0)}\leq C_N h^N$.  
We note 
\[
\tilde B(0)= \y\bigpare{\tfrac{hr}{\e}} \I \I^* A_0 \I \I^* \y\bigpare{\tfrac{hr}{\e}} 
=A_0+ O(h^N)
\]
by the choice of $\e>0$. Combining these, we conclude the assertion. 
\end{proof}


\section{Proof of Theorems \ref{thm:WF} and \ref{thm:MZ}}

Let $(r_0,\rho_0,\th_0,\o_0)\in T^*(\re_+\times \pa M)$, and suppose $\o_0\neq 0$
as in the last section. Also we let $a\in C_0^\infty(T^*(\re_+\times \pa M))$ be supported in 
a small neighborhood of $(r_0,\rho_0,\th_0,\o_0)$ and we set
\[
A_0=Op^W(a^h), \quad a^h(r,\rho,\th,\o)=a(hr,\rho,\th,h\o). 
\]
Let $\e>0$ also as in the last section. We write 
\[
(r_\pm,\rho_\pm,\th_\pm,\o_\pm)=w_{c,\pm}^{-1}(r_0,\rho_0,\th_0,\o_0)
\]
as in Section~2, and we recall $w_{c,\pm}$ are diffeomorphisms from 
$\re\times \re_\pm\times (T^*\pa M\setminus 0)$ to $\re_+\times \re\times (T^*\pa M\setminus 0)$. 
We also note 
\[
E_0=p_c(r_0,\rho_0,\th_0,\o_0)=\frac12\rho_\pm^2 >0
\]
by the conservation of the energy. 

\begin{lem}
If $\d>2\e^2$, then 
\[
\wlim_{t\to\pm\infty} \y(P_f/\d)A(t) \y(P_f/\d) = \y(P_f/\d) W_\pm^* A_0 W_\pm \y(P_f/\d).
\]
\end{lem}

\begin{proof}
It is easy to show by the stationary phase method that 
\[
\slim_{t\to\pm\infty} \bigpare{1-\y\bigpare{\tfrac{hr}{\e\jap{t}}}}\mathcal{I} e^{-itP_f/h} \y(P_f/\d)=0
\]
(for fixed $h$), since the stationary points (in $\rho$) satisfy $hr=t \rho$. 
Now this implies 
\[
\slim_{t\to\pm\infty} e^{itP/h}Y\mathcal{I} e^{-itP_f/h}\y(P_f/\d)
=W_\pm \y(P_f/\d)
\]
and the claim follows immediately. 
\end{proof}

This implies, combined with Lemmas 4.1 and 4.2:

\begin{lem}
Let $A_0$ as above. Then $W_\pm^* A_0 W_\pm$ are pseudodifferential operators with 
the symbols $b^h_\pm$ given in Lemma~4.1. In particular, $\S(W_\pm^* A_0 W_\pm)$ 
are supported in $w_{c,\pm}^{-1}[\supp a^h]$ modulo $O(h^\infty)$-terms, and the 
principal symbols (modulo $S(h^\m,g_1^h)$) are given by $a^h\circ w_{c,\pm}$. 
\end{lem}

For the moment, we set 
\[
\rho_0=0,\quad \text{and hence}\quad r_\pm=0.
\]
Then we may take $\e=\sqrt{E_0}$ provided $a$ is supported in a sufficiently small 
neighborhood of $(r_0,0,\th_0,\o_0)$. 

Now let us consider $(0,\rho_-,\th_-,\o_-)$ (with $\o_-\neq 0$, $\rho_->2\e$) 
is given, and $(0,\rho_0,\th_0,\o_0)$ is set by 
$w_{c,-}(0,\rho_-,\th_-,\o_-)= (0,\rho_0,\th_0,\o_0)$.
The converse of Lemma~5.2 is given as follows: 

\begin{lem}
Let $\tilde a\in C_0^\infty(\re\times \re_-\times (T^*\pa M\setminus 0))$ be supported in 
a small neighborhood of $(0,\rho_-,\th_-,\o_-)$, and let 
\[
\tilde A= Op^W(\tilde a^h), \quad \tilde a^h(r,\rho,\th,\o)=\tilde a(hr,\rho,\th,h\o).
\]
Then there is a symbol $a_0^h$ supported in 
$w_{c,-}[\supp \tilde a^h]$ such that for any $f\in C_0^\infty(\re_+)$, 
\[
f(P)A_0 f(P) = W_- f(P_f) \tilde A f(P_f) W_-^*,
\]
where $A_0=Op^W(a_0^h)$. 
Moreover, the principal symbol (modulo $S(h^\m,g_1^h)$) is give by $\tilde a^h\circ w_{c,-}^{-1}$. 
\end{lem}

\begin{proof}
We set $a_{0,0}^h =\tilde a^h\circ w_{c,-}^{-1}$. 
Then by Lemma 5.2, we have 
\[
a^h_{-,1} :=\S (\tilde A -W_-^* Op^W(a^h_{0,0})W_-)\in S(h^\m, g_1^h),
\]
and it is supported in $\supp[\tilde a^h]$ modulo $O(h^\infty)$-terms. 
Then we set $a_{0,1}=a^h_{-,1}\circ w_{c,-}^{-1}$, and set
\[
a_{-,2}^h := \S(\tilde A-W_-^* Op^W(a_{0,0}^h+a_{0,1}^h) W_-)\in S(h^{2\m},g_1^h).
\]
We construct $a_{-,j}^h$, $j=2,3,\dots$, iteratively by 
\[
a_{-,j}^h := \S(\tilde A-W_-^* Op^W(a_{0,0}^h +\cdots + a_{0,j-1}^h) W_-)\in S(h^{j\m},g_1^h),
\]
$a_{0,j}^h=a_{-,j}^h\circ w_{c,-}^{-1}$, and we set 
$a_0^h\sim \sum_{j=0}^\infty a_{0,j}^h$ as an asymptotic sum. Then we have 
\[
\tilde A= W_-^* A_0 W_- 
\]
modulo $S(h^\infty\jap{r}^{-\infty}\jap{\o}^{-\infty},g_1)$-terms. 
We note, since there are no positive eigenvalues
(see \cite{IS}. See also \cite{MZ}), we also have $W_\pm f(P_f) W_\pm^*=f(P)$
by virtue of the intertwining property and the asymptotic completeness (\cite{IN2}).
These imply 
\[
W_-f(P_f) \tilde A f(P_f) W_-^* = W_- f(P_f) W_-^* A_0 W_- f(P_f) W_-^*
= f(P) A_0 f(P),
\] 
and this implies the assertion. 
\end{proof}

We note Lemma 5.3 naturally holds for $w_{c,+}$ instead of $w_{c,-}$, but we only 
use the above case. By Lemma 5.3, we learn 
\[
S f(P_f) \tilde Af(P_f)  S^* = W_+^* f(P) A_0 f(P) W_+ = f(P_f) (W_+^* A_0 W_+)f(P_f).  
\]
By Lemma~5.2, $W_+^* A_0 W_+$ is a pseudodifferential operator. 
By choosing $f\in C_0^\infty(\re_+)$ so that $f(\rho^2/2)=1$ in a neighborhood 
of the support of $\tilde a$, we may omit $f(P_f)$ factors up to negligible terms. 
Thus, $S \tilde A S^*$ is 
a pseudodifferential operator with the symbol supported in $s_c[\supp\tilde a^h]$, 
and the principal symbol is given by $\tilde a^h\circ s_c^{-1}$, where $\tilde a$ is 
the symbol given in Lemma~5.3, i.e., $\tilde a$ is supported in a small neighborhood 
of $(0,\rho_-,\th_-,\o_-)$. 

We note, by the intertwining property of the scattering operator, 
\[
e^{-itP_f} S = S e^{-itP_f}, \quad \forall t\in\re. 
\]
This in turn implies 
\[
T_\t S = S T_\t, \quad \forall\t\in\re, \quad\text{where } T_\t=\exp(-i\t\sqrt{2P_f}).
\]
On the other hand, $\sqrt{2P_f}=\mp i\tfrac{\pa}{\pa r}$ on $\mathcal{H}_{f,\pm}$, 
and hence $T_\t$ are translations with respect to $r$. More precisely, 
we have 
\[
T_\t u_\pm(r,\th)= u_\pm (r\mp \t,\th)
\quad \text{for }u_\pm \in\mathcal{H}_{f,\pm}.
\]
We learn from these that 
\[
S(T_\t\tilde A T_\t^* ) S^*
=T_\t (S\tilde A S^*) T_\t^*,
\]
and that the symbols of $T_\t \tilde A T_\t^*$ and $T_\t (S\tilde A S^*) T_\t^*$ are 
given by $\tilde a^h(r+\t,\rho,\th,\o)$ and $\S(S \tilde A S^*)(r+\t,\rho,\th,\o)$, respectively. 
Using this observation, we may replace $\tilde a$ by a symbol supported in a small 
neighborhood $(r_-,\rho_-,\th_-,\o_-)$ with arbitrary $r_-\in\re$. Thus we have proved: 

\begin{lem}
Let $a\in C_0^\infty(\re\times \re_-\times (T^*\pa M\setminus 0))$ be supported 
in a small neighborhood of $(r_-,\rho_-,\th_-,\o_-)$ with $|\rho_-|\geq 2\e$, and let
\[
\tilde A= Op^W(a^h), \quad a^h(r,\rho,\th,\o)=a(hr,\rho,\th,h\o).
\]
Then $S\tilde A S^*$ is a pseudodifferential operator with a symbol 
supported in $s_c[\supp a^h]$ modulo $O(h^\infty)$-terms, and 
the principal symbol (modulo $S(h^\m,g_1^h)$) is given by $a^h\circ s_c^{-1}$. 
\end{lem}

Here we have used the formula: 
\[
s_c(r,\rho,\th,\o)= (-r,-\rho,\exp(\pi H_{\sqrt{2q}})(\th,\o)).
\]

We set $\hat{\mathcal{H}}_{f,\pm}= \mathcal{F}\mathcal{H}_{f,\pm}$. 
Then $\mathcal{F} S \mathcal{F}^{-1}$ is a unitary map from $\hat{\mathcal{H}}_{f,-}$ 
to $\hat{\mathcal{H}}_{f,+}$. 
For notational simplicity, we set
\[
\Pi u(r,\th)= u(-r,\th)\quad \text{for }u\in \mathcal{H}_{f,\pm}, 
\]
so that  $\mathcal{F} (S\Pi) \mathcal{F}^{-1}$ is a unitary map on $\hat{\mathcal{H}}_{f,+}$. 
By the intertwining property above, $\mathcal{F} (S\Pi) \mathcal{F}^{-1}$ commutes with 
functions of $\rho$, and hence is decomposed so that  
\[
\mathcal{F} (S\Pi) \mathcal{F}^{-1} u(\rho,\o) = (S(\rho^2/2)u(\rho,\cdot))(\o)
\quad \text{on } 
\hat{\mathcal{H}}_{f,+}\cong L^2(\re_+; L^2(\pa M)), 
\]
where $S(\l)\in B(L^2(\pa M))$ is the scattering matrix.

\begin{proof}[Proof of Theorem \ref{thm:WF}]
We recall the semiclassical type  characterization of the wave front set: 
Let $g(\rho,\th)\in \mathcal{D}'(\re_+\times \pa M)$, and let 
$(\rho_0,\th_0,r_0,\o_0)\in T^*(\re_+\times \pa M)$. 
$(\rho_0,\th_0,r_0,\o_0)\notin WF(g)$ if and only if there is 
$a\in C_0^\infty(T^*(\re_+\times \pa M))$ such that $a(\rho_0,\th_0,r_0,\o_0)\neq 0$ 
and 
\[
\bignorm{a(\rho,\th,h D_\rho,h D_\th)g} = O(h^\infty)\quad \text{as }h\to+0.
\]
We may replace $a$ by an $h$-dependent symbol with a principal symbol 
which does not vanish at $(\rho_0,\th_0,r_0,\o_0)$. 

We fix $\l_0=\rho_0^2/2$ with $\rho_0>2 \e$ and consider $S(\l)$ where 
$\l$ is in a small neighborhood of $\l_0$. 
Let $u\in L^2(\pa M)$ and let $v\in C_0^\infty(\re_+)$ supported in a small 
neighborhood of $\l_0$. Then it is easy to see 
\[
WF(v(\rho)u(\th))=\{(\rho,\th,0,\o)\,|\, \rho\in\supp v, (\th,\o)\in WF(u)\}.
\]
Then, by Lemma 5.4 and the above characterization of the wave front set, 
we learn 
\begin{align*}
&WF(\mathcal{F}(S\Pi)\mathcal{F}^{-1} v(\rho)u(\th))
= (1\otimes \exp(\pi H_{\sqrt{2q}}))WF(v(\rho)u(\th)) \\
&\qquad = \{(\rho,\th,0,\o)\,|\, \rho\in\supp v, (\th,\o)\in \exp(\pi H_{\sqrt{2q}})WF(u)\}
\end{align*}
(cf. \cite{Na1}). 
By the definition of the scattering matrix, this implies 
\[
WF(S(\l)u)\subset \exp(\pi H_{\sqrt{2q}})WF(u)
\]
for $\l\in \supp v$. Since this argument works for $S^{-1}$ also, the above 
inclusion is actually an equality, and we conclude Theorem~\ref{thm:WF}. 
\end{proof}

\begin{proof}[Proof of Theorem \ref{thm:MZ}]
Here we suppose $\m=1$. Then by Lemma 5.4 and the Beals-type characterization 
of FIOs (Appendix B, Theorem~B.1), we learn $\mathcal{F}(S\Pi)\mathcal{F}^{-1} $ is an FIO 
associated to $1\otimes \exp(\pi H_{\sqrt{2q}})$ on $\bigset{(\rho,\th,r,\o)}{\o\neq 0}$. 
Since $\mathcal{F}(S\Pi)\mathcal{F}^{-1}$ is decomposed to $\{S(\l)\}$, this implies 
$S(\l)$ are FIOs on $\pa M$ associated to the canonical transform 
$\exp(\pi H_{\sqrt{2q}})$ (cf. Appendix B, Proposition~B.4). 
\end{proof}


\section{Proof of Theorem \ref{thm:GFIO}}

Here we discuss how to generalize the proof of Theorem~\ref{thm:MZ} to conclude 
Theorem~\ref{thm:GFIO}. 

We first modify the Egorov theorem type argument in Section~4. 
Let $(r_0,\rho_0,\th_0,\o_0)\in T^*M_\infty$, $\o_0\neq 0$, and let $\O_0$ be a small 
neighborhood of $(r_0,\rho_0,\th_0,\o_0)$. We suppose $a\in C_0^\infty(T^*M_\infty)$ 
is supported in $\O_0$, and we consider the behavior of $A(t)$ as in Section~4. We set 
\[
w^*(t)= \exp(-itH_{p_f})\circ \exp(tH_{p}),
\]
which is well-defined for  $X\in T^*M_\infty$ as long as $\exp(tH_p)(X)\in T^*M_\infty$. 
By the discussion in the proof of Theorem~2.2, this condition is always 
satisfied if $X=(r,\rho,\th,\o)\in \O_0^h$ and $h$ is sufficiently small. 
We set 
\[
w(t)=w^*(t)^{-1} =\exp(-tH_p)\circ \exp(tH_{p_f})
\]
on the range of $w(t)$. We note 
\[
w^*_\pm= \lim_{t\to\pm\infty} w^*(t)
\]
on $\O_0^h$ with sufficiently small $h$, and 
\[
w_\pm= \lim_{t\to\pm\infty} w(t)
\]
on $w_{\pm}^{-1}[\O_0^h]$ with sufficiently small $h$. Convergence of these maps holds 
in the $C^\infty$-topology. 

We replace Lemma~4.1 by the following slightly different statement: 

\begin{lem}
There exists $b^h(t;r,\rho,\th,\o)\in C_0^\infty(T^*M_f)$ such that 
\begin{enumerate}
\renewcommand{\theenumi}{\roman{enumi}}
\renewcommand{\labelenumi}{{{\rm (\theenumi)} }}
\item $b^h(0)=a^h$. 
\item $b^h(t)$ is supported in $w^*(t)\bigbrac{\supp a^h}$. 
\item $b^h(t)\in S(1,g_1^h)$, and it is bounded uniformly in $t\in\re$. 
\item $b^h(t)-a^h\circ w(t)\in S(h,g_1^h)$, i.e., the principal symbol of $b^h(t)$ 
is given by $a^h\circ w(t)$, and the remainder is bounded uniformly in $t$. 
\item If we set $B(t)=Op^W(b^h(t))$, then 
\[
\Bignorm{\frac{d}{dt}B(t)+\frac{i}{h}[L(t),B(t)]} \leq C_N \jap{t}^{-1-\m} h^N, \quad h>0, 
\]
with any $N$.
\item $B(t)$ converges to $B_\pm$ as $t\to\pm\infty$ in $B(\mathcal{H}_f)$, and 
the symbols : $b^h_\pm :=\S(B_\pm)$ satisfy 
\[
b^h_\pm-a^h\circ w_\pm \in S(h,g_1^h).
\]
\end{enumerate}
\end{lem}

We note that $w(t)$ is not homogeneous in $(r,\o)$-variables, but very close to 
a homogeneous map when $|(r,\o)|$ is very large thanks to Theorem~2.2. 

In order to prove Lemma~6.1, we set 
\[
b_0^h(t)=a^h\circ w(t) =a\circ \exp(-tH_p)\circ \exp(tH_{p_f}),
\]
which is supported in $w^*(t)[\O_0^h]$. We have $b_0^h(t)\in S(1,g_1^h)$ uniformly in $t$ 
(for small $h$) again by Theorem~2.2. Moreover, $b_0^h$ satisfies 
\[
\frac{\pa}{\pa t} b_0^h(t) =-h^{-1}\{\ell(t), b_0^h(t)\}
\]
where $\ell(t)=\S(L(t))$. Hence the first remainder term $r_0^0(t)$ (as defined in Section~4)
satisfies 
\[
\bigabs{\pa_r^\a\pa_\rho^\b\pa_\th^\c\pa_\o^\d r^0_0(t)}\leq C\jap{t}^{-1-\m-|\a|}
h^{1+|\a|+|\d|}
\]
with any indices  $\a,\b,\c,\d$. Then we construct the asymptotic solution as in the proof of 
Lemma~4.1 by solving transport equations:
\[
\frac{\pa}{\pa t} b_j^h(t) +h^{-1} \{ \ell(t), b_j^h(t)\} = -r_j^h(t), \quad j=0,1,2,\dots, 
\]
and we conclude Lemma~6.1.\qed 

Lemma~4.2 holds when the construction of  $B(t)$ is replaced as above, with no modifications. 
Lemmas~5.2 and 5.3 holds in the following form. The proofs are the same. 

\begin{lem}
Let $A_0$ as above. Then $W_\pm^* A_0 W_\pm$ are pseudodifferential operators with 
the symbols $b^h_\pm$ given in Lemma~6.1. In particular, $\S(W_\pm^* A_0 W_\pm)$ 
are supported in $w_{\pm}^{-1}[\supp a^h]$ modulo $O(h^\infty)$-terms, and the 
principal symbols (modulo $S(h,g_1^h)$) are given by $a^h\circ w_{\pm}$. 
\end{lem}

\begin{lem}
Let $\tilde a\in C_0^\infty(\re\times \re_-\times (T^*\pa M\setminus 0))$ be supported in 
a small neighborhood of $(0,\rho_-,\th_-,\o_-)$, and let 
\[
\tilde A= Op^W(\tilde a^h), \quad \tilde a^h(r,\rho,\th,\o)=\tilde a(hr,\rho,\th,h\o).
\]
Then $W_-\tilde A W_-^*$ is a pseudodifferential operator with a symbol supported in 
$w_{-}[\supp \tilde a^h]$, and the principal symbol (modulo $S(h,g_1^h)$) is 
give by $\tilde a^h\circ w_{-}^{*}$. 
\end{lem}

Combining these, we learn (as in Section~5) the following assertion. 

\begin{lem}
Let $a\in C_0^\infty(\re\times \re_-\times (T^*\pa M\setminus 0))$ be supported 
in a small neighborhood of $(r_-,\rho_-,\th_-,\o_-)$ with $|\rho_-|\geq 2\e$, $\o_-\neq 0$, and let
\[
\tilde A= Op^W(a^h), \quad a^h(r,\rho,\th,\o)=a(hr,\rho,\th,h\o).
\]
Then $S\tilde A S^*$ is a pseudodifferential operator with a symbol 
supported in $s[\supp a^h]$ modulo $O(h^\infty)$-terms, and 
the principal symbol (modulo $S(h,g_1^h)$) is given by $a^h\circ s^{-1}$. 
\end{lem}

In the following, we consider $(r,\rho,\th,\o)\in \O_0^h$ with some $\O_0$ and 
sufficiently small $h$, or equivalently, when $|\o|$ is sufficiently large. 
By the conservation of energy (or equivalently, by the invariance in the shift in $r$), 
the classical scattering operator has the following form:
\begin{equation}\label{eq:cl-sc}
s(r,\rho,\th,\o) = (-r+g(\rho,\th,\o), -\rho, s(\l)(\th,\o)),
\end{equation}
where $\l=\rho^2/2$ and $s(\l)$ is a canonical transform on $T^*\pa M$ for each $\l>0$. 
(We note that without $g(\rho,\th,\o)$, the map $s$ is not necessarily canonical.) 
Moreover, by Theorem~2.1, we have for any indices $\a,\b,\c$, 
\begin{align*}
&\bigabs{\pa_\rho^\a\pa_\th^\b\pa_\o^\c g(\rho,\th,\o)}\leq C h^{-1+\m+|\c|}, \\
&\bigabs{\pa_\rho^\a\pa_\th^\b\pa_\o^\c s_1(\rho,\th,\o)}\leq C h^{\m+|\c|}, \\
&\bigabs{\pa_\rho^\a\pa_\th^\b\pa_\o^\c s_2(\rho,\th,\o)}\leq C h^{-1+\m+|\c|}, 
\end{align*}
on $\O_0^h$, where $\O_0$ is a small neighborhood of $(0,\rho_-,\th_-,\o_-)$, and 
$s_1, s_2$ are defined by 
\[
(s_1(\rho,\th,\o),s_2(\rho,\th,\o))= s(\l)(\th,\o)-\exp(\pi H_{\sqrt{2q}})(\th,\o),
\]
i.e., $s_1$ denotes the $\th$-components of the RHS terms, and $s_2$ denotes 
the $\o$-components. 
These estimates imply $s$ is an asymptotically homogeneous (in $(r,\o)$-variables) 
in the sense of \cite{IN3}  Section~4. 

In general, an operator $U$ with the distribution kernel $u$ is called an FIO of order $m$ 
associated to an asymptotically homogeneous canonical transform $S$ if $u$ is a 
Lagrangian distribution associated to 
\[
\S_S:=\bigset{(x,y,\x,-\y)}{(x,\x)=S(y,\y)}, 
\]
i.e., for any $a_1,\dots, a_N\in S^1_{cl}$ such that $a_j$ vanishes on $\S_S$ for each $j$, 
$Op(a_1)\cdots Op(a_N)u\in B_{2,\infty}^{-m-n/2,\infty}(\re^{2n})$ (\cite{IN3}). 
The Beals type characterization of FIOs discussed in Appendix~B hold for such FIOs 
without any change. 

By Lemma~6.4 and the analogue of Corollary~B.2, we learn $S$ is an FIO associated 
to the classical scattering map $s$. Moreover, by Proposition~B.4, we learn that 
the scattering matrix $S(\l)$ is an FIO associated to $s(\l)$, where $s(\l)$ is defined 
by \eqref{eq:cl-sc} and it is asymptotic to $\exp(\pi H_{\sqrt{2q}})$. Thus we have proved 
the following slightly more precise version of Theorem~1.3:

\begin{thm}
Suppose Assumption~A. Then for each $\l>0$, $S(\l)$ is an FIO associated 
to $s(\l)$ defined by \eqref{eq:cl-sc}. The canonical map $s(\l)$ on $T^*\pa M$ 
is asymptotically homogeneous in $\o$, asymptotic to $\exp(\pi H_{\sqrt{2q}})$ 
with the error of $O(|\o|^{1-\m})$. 
\end{thm}

\appendix
\section{Local decay estimates}
Let $P$  be as in Section~1. 
For a symbol $a$, we denote $a^h(r,\rho,\th,\o)= a(hr,\rho,\th,h\o)$. 
Then we have the following:

\begin{thm}
Let $(r_0,\rho_0,\th_0,\o_0)\in T^*M_\infty\cong T^*\re_+\times T^*\pa M$, 
and suppose $\o_0\neq 0$. We denote the $\e$-neighborhood of 
$(r_0,\rho_0,\th_0,\o_0)$ by $\O_\e$. 
We suppose $\e>0$ so small that $\O_{2\e}\Subset T^*\re_+\times (T^*\pa M\setminus 0)$. 
If $a\in C_0^\infty(T^*M_\infty)$ is real-valued, and supported in 
$\O_\e$,  then there is an $h$-dependent symbol: 
$b(t)\in C_0^\infty(T^*M_\infty)$ for any $t\in\re$ such that 
\begin{enumerate}
\renewcommand{\theenumi}{\roman{enumi}}
\renewcommand{\labelenumi}{{{\rm (\theenumi)} }}
\item $|a(r,\rho,\th,\o)|\leq c_1 b(0;r,\rho,\th,\o)$ with some $c_1>0$. 
\item $b(t)$ is supported in $\O(t):=\exp(tH_{p_c})[\O_{2\e}]$ for $t\in\re$. 
\item For any indices $\a,\b,\c$ and $\d$, there is $C_{\a\b\c\d}>0$ such that 
\[
\bigabs{\pa_r^\a\pa_\rho^\b \pa_\th^\c \pa_\o^\d b(t,r,\rho,\th,\o)}
\leq C_{\a\b\c\d}, \quad (r,\rho,\th,\o)\in T^*M_\infty, t\in\re.
\]
\item There is $R(t)\in B(L^2(M))$ such that $\norm{R(t)}\leq C_N h^N$ 
for any $N$, and 
\[
e^{-itP/h} Op^W(a^h) e^{itP/h} \leq c_1 Op^W(b^h(t)) + R(t)
\]
for $t>0$, and the reverse inequality for $t<0$. Moreover, $R(t)$ satisfies 
\[
\bignorm{K^N R(t) K^N}_{B(L^2)}\leq C_N h^N, \quad t\in\re
\]
for any $N$, where $K(\cdot)=\jap{\dist(\cdot,\O^h(t))}$ with 
\[
\supp[b^h(t)] \subset \O^h(t):=\bigset{(r,\rho,\th,\o)}{(hr,\rho,\th,h\o)\in \O(t)}.
\] 
\end{enumerate}
\end{thm}

Before proving Theorem A.1, we present a corollary which is needed in 
Section 4. 

\begin{cor}
Let $\bar\y\in C^\infty(\re)$ be such that $\bar\y(r)=0$ if $r>2$, and 
$\bar\y(r)=1$ if $r\leq 1$. We choose  $\e_1>0$ so small that 
\[
\dist\bigpare{\bigset{(r,\rho,\th,\o)}{|r|\leq \e_1\jap{t}},\O(t)}
\geq \d\jap{t}
\]
with some $\d>0$. Then for any $N$ there is $C_N>0$ such that 
\[
\bignorm{\bar\y\bigpare{\tfrac{hr}{\e_1\jap{t}}}e^{-itP/h} Op(a^h)}
\leq C_N h^N \jap{t}^{-N}, \quad t\in\re.
\]
\end{cor}

We note if $\e>0$ is chosen sufficiently small, then we can find $\e_1>0$ 
satisfying the above property. 

\begin{proof}[Proof of Corollary A.2]
We apply Theorem A.1 with $\tilde a$ such that $Op^W(\tilde a)=
Op^W(a) Op^W(a)^*$, which satisfies the same condition. 
Then we have 
\begin{align*}
\bigabs{\bar\y\bigpare{\tfrac{hr}{\e_1\jap{t}}}e^{-itP/h} Op(a^h)}^2 
&={\bar\y\bigpare{\tfrac{hr}{\e_1\jap{t}}}e^{-itP/h} Op(\tilde{a}^h) e^{itP/h}
\bar\y\bigpare{\tfrac{hr}{\e_1\jap{t}}}}\\
&\leq c_1 {\bar\y\bigpare{\tfrac{hr}{\e_1\jap{t}}}Op(b^h(t))
\bar\y\bigpare{\tfrac{hr}{\e_1\jap{t}}}}
+{ \bar\y\bigpare{\tfrac{hr}{\e_1\jap{t}}} R(t)\bar\y\bigpare{\tfrac{hr}{\e_1\jap{t}}} }\\
&\leq C_N h^N\jap{t}^{-N},
\end{align*}
where we have used the fact that $\supp[b^h(t)]$ is separated from 
$\O^h(t)$ with the distance bounded from 
below by $\d \jap{t/h}$. 
\end{proof}

\begin{proof}[Proof of Theorem A.1] 
The proof is analogous to that of \cite{Na1}, \cite{It} and \cite{IN1} Section~3, 
and we only sketch the main steps. 
We may suppose $a$ is non negative without loss of generality. 
If we set 
\[
\g(t)=a\circ \exp(tH_{p_c})^{-1},
\]
then it is easy to see 
\[
\frac{\pa}{\pa t} \g=-\{p_c,\g\}, \quad \g(0)=a,
\]
and this is a good candidate for the principal term of $b(t)$, 
but $\g$ does not satisfy the boundedness of the derivatives uniformly in $t$. 
We choose $\f\in C_0^\infty(\re)$ so that 
\[
\supp\f \subset [-1,1], \quad \f(t)\geq 0 \text{ for all $t$}, \quad 
\int_{-1}^1 \f(t)dt =1, 
\]
and moreover,  $\pm \f'(t)\leq 0$ for $\pm t\geq 0$. We set
\[
\f_\n(t)= \f(t/\n), \quad \text{for }\n>0, 
\]
and we denote the convolution in the $t$-variable by ``$\underset{t}{*}$''. 
Then we set
\[
b_0(t,\cdot)=\f_{\d\jap{t}}\underset{t}{*} \g=\int \f_{\d\jap{t}}(t-s)\g(s,\cdot)ds
\]
with sufficiently small $\d>0$. Then we have 
\begin{align}
\frac{\pa}{\pa t} b_0 &= \int \pa_t(\f_{\d\jap{t}}(t-s)) \g(s,\cdot)ds \\
&= -\int \frac{t(t-s)}{\d\jap{t}^3}\f'((t-s)/\d\jap{t})\g(s,\cdot)ds 
 +\f_{\d\jap{t}}\underset{t}{*} (\pa_t\g)\nonumber \\
&\geq -\f_{\d\jap{t}}\underset{t}{*} \{p_c,\g\}=-\{p_c,b_0(t,\cdot)\}\nonumber
\end{align}
for $t>0$ by the condition of $\f$. We have the reverse inequality for $t<0$. 

We then show the derivatives of $b_0$ satisfies the required uniform boundedness. 
We first note 
\[
\tilde\g(t;r,\rho,\th,\o) :=\g(t;r+t\rho,\rho,\th,\o) \to a\circ w_\pm
\quad (t\to\pm\infty)
\]
in the $C_0^\infty$-topology, by virtue of the existence of the classical scattering
for $p_c$. Thus we have the representation:
\[
\g(t;r,\rho,\th,\o)= \tilde \g(t;r-t\rho,\rho,\th,\o)
\]
with $\tilde \g(t)$ uniformly bounded in $C_0^\infty(T^*M)$. Hence we learn that
the derivatives in variables except for $\rho$ are uniformly bounded. 
Then this property applies also for $b_0(t)$. 
Let us consider the first derivative of $b_0(t)$ in $\rho$:
\begin{align*}
\pa_\rho b_0(t)&= -\int\f_{\d\jap{t}}(t-s) s(\pa_r\tilde\g)(s,r-s\rho,\rho,\th,\o)ds \\
&\quad + \int\f_{\d\jap{t}}(t-s) (\pa_\rho\tilde\g)(s,r-s\rho,\rho,\th,\o)ds.
\end{align*}
The second term is clearly uniformly bounded. We note
\[
(\pa_r\tilde\g)(s;r-s\rho,\rho,\th,\o) =
-\frac{1}{\rho}\biggbra{\frac{\pa}{\pa s}[\tilde\g(s;r-s\rho,\rho,\th,\o)]
-(\pa_s\tilde\g)(s;r-s\rho,\rho,\th,\o)}
\]
and then by integration by parts we have:
\begin{align*}
&\int\f_{\d\jap{t}}(t-s) s(\pa_r\tilde\g)(s,r-s\rho,\rho,\th,\o)ds \\
&\quad = \frac{1}{\rho}\int\frac{\pa}{\pa s}\bigbra{\f((t-s)/\d\jap{t})s}
\tilde\g(s,r-s\rho,\rho,\th,\o)ds \\
&\quad\quad  + \frac{1}{\rho}\int \f((t-s)/\d\jap{t})s (\pa_s\tilde\g)(s;r-s\rho,\rho,\th,\o)ds\\
&\quad = \frac{1}{\rho}\int \f((t-s)/\d\jap{t})\tilde\g(s,r-s\rho,\rho,\th,\o)ds \\
&\quad\quad -\frac{1}{\rho}\int \frac{s}{\d\jap{t}}
\f'((t-s)/\d\jap{t})\tilde\g(s,r-s\rho,\rho,\th,\o)ds \\
&\quad\quad  + \frac{1}{\rho}\int \f((t-s)/\d\jap{t})s (\pa_s\tilde\g)(s;r-s\rho,\rho,\th,\o)ds
\end{align*}
Each term in the last expression is bounded uniformly in $t$ since $s\sim t$, 
and $\pa_s\tilde\g =O(\jap{s}^{-2})$. Repeating this procedure, we can show 
that all the derivatives of $b_0$ are uniformly bounded. It is also 
easy to check that $b_0$ satisfies the required support property 
provided $a$ is supported in a sufficiently small neighborhood, 
and $\d>0$ is chosen sufficiently small. 

Now by (A.1) and the sharp G{\aa}rding inequality, we have 
\[
\frac{d}{dt} Op^W(b_0^h(t)) \geq -\frac{i}{h}[ P,Op^W(b_0^h(t))] + Op(r_1^h(t))
\]
with $r_1(t)=O(h^\m)$. We set $c_j=7/4-2^{-j}$ for $j=1,2,\dots$, and set
\[
a_j(r,\rho,\th,\o)=a\biggpare{\frac{r}{c_j},\frac{\rho}{c_j},\frac{\th}{c_j},\frac{\o}{c_j}}, 
\quad b_j(t) = \f_{\d\jap{t}}\underset{t}{*} (a_j\circ \exp(tH_{p_c})).
\]
Then we set 
\[
b(t)\sim b_0(t)+ \sum_{j=1}^\infty \m_j b_j(t),
\]
with appropriately chosen constants $\m_j>0$ so that 
\[
\frac{d}{dt} Op^W(b^h(t))\geq -\frac{i}{h}[P,Op^W(b^h(t))] + O(h^\infty), 
\]
and $b(t)$ satisfies all the required properties. We refer to \cite{Na1} and 
\cite{IN1} for the detail of the above construction. 
\end{proof}

\section{Beals type characterization of Fourier integral operators}
In this appendix, we consider operators on $\re^n$, and we discuss 
Beals type characterization of FIOs in terms of $h$-pseudodifferential operators. 
We use the result for scattering manifolds, but the generalization 
is straightforward, and we omit it. 
Most of the arguments here are similar to \cite{IN3} Section~2, and we mainly discuss the 
modifications necessary to show our results. 

We let $S$ be a canonical diffeomorphism on $T^*\re^n$, which is also supposed to be 
homogeneous in $\x$-variables, i.e., 
\[
\text{if } (y,\y)=S(x,\x), \text{ then } S(x,\l\x)=(y,\l\y) \text{ for }\l>0.
\]
We also let $U\in \mathcal{L}(\mathcal{S},\mathcal{S}')$, 
and let $u\in \mathcal{D}'(\re^{2n})$ be its distribution kernel. 
For a symbol $a\in C^\infty(T^*\re^n)$, we denote
\[
a^h(x,\x) =a(x,h\x), \quad Op^W(a^h)= a^W(x, hD_x)
\]
for $h>0$ as before. For $a\in C_0^\infty(T^*\re^n\setminus 0)$, we define
\[
Ad_S(a^h)U= Op^W(a^h\circ S^{-1})U- UOp^W(a^h)
\in \mathcal{L}(\mathcal{S},\mathcal{S}').
\]
We note $Op^W(a^h\circ S^{-1})= Op^W((a\circ S^{-1})^h)$ since $S$ is homogeneous 
in $\x$. 

\begin{thm}
Let $U\in B(L^2_{cpt}(\re^n),L^2_{loc}(\re^n))$. 
Suppose for any $a_1,a_2,\dots, a_N\in C_0^\infty(T^*\re^n\setminus 0)$, 
there is $C_N>0$ such that
\begin{equation}\label{cond:B1}
\bignorm{Ad_S(a_1^h)Ad_S(a_2^h)\cdots Ad_S(a_N^h)U}_{B(L^2)}\leq C_N h^N.
\end{equation}
Then $U$ is an FIO of order 0 associated to $S$. 
\end{thm}

\begin{cor}
Let $S$ and $U$ as above. If for any $a\in C_0^\infty(T^*\re^n\setminus 0)$
there is an $h$-dependent symbol $b\in C_0^\infty(T^*\re^n\setminus 0)$ 
such that: 
\begin{align*}
&\bigabs{\pa_x^\a\pa_\x^\b b(h;x,\x)}\leq C_{\a\b}h, 
\quad \text{for any $\a,\b\in\ze_+^n$, }h\in(0,1],\\
&Ad_S(a^h) U= Op^W(b^h)U +R, \quad \norm{R}_{B(L^2)}=O(h^\infty),
\end{align*}
then $U$ is an FIO of order 0 associated to $S$. 
\end{cor}

\begin{proof}[Proof of Corollary B.2]
We show \eqref{cond:B1} follows from the above condition. 
The cases $N=0,1$ are obvious. Let $N=2$ and we write
\[
Ad_S(a_j^h) U= Op^W(b_j^h)U +R_j, \quad j=1,2. 
\]
Then we have 
\begin{align*}
&Ad_S(a_1^h)Ad_S(a_2^h) U \\
&= Op^W(a_1^h\circ S^{-1}) Op^W(b_2^h)U 
-Op^W(b_2^W)UOp^W(a_1^h) +Ad_S(a_1^h)R_2\\
&= [Op^W(a_1^h\circ S^{-1}),Op^W(b_2^h)]U +Op^W(b_2^h)Op^W(b_1^h) U \\
&\quad +Ad_S(a_1^h)R_2 + Op^W(b_2^h)R_1 \\
&=Op^W(b_{12}^h) U + R_{12},
\end{align*}
where $R_{12}=O(h^\infty)$ and $b_{12}\in C_0^\infty(T^*\re^n\setminus 0)$
satisfies
\[
\bigabs{\pa_x^\a\pa_\x^\b b_{12}(h;x,\x)}\leq C'_{\a\b}h^{2}, 
\quad \text{for any $\a,\b\in\ze_+^n$, }h\in(0,1],
\]
and \eqref{cond:B1} for $N=2$ follows. Iterating this procedure, we
obtain \eqref{cond:B1} for any $N$. 
\end{proof}

In order to prove Theorem~B.1, we first note the semiclassical type 
characterization of Besov spaces. 
By the standard partition-of-unity argument, it is straightforward to observe that 
$u\in B_{2,loc}^{\s,\infty}(\re^m)$ if and only if for any $(x_0,\x_0)\in T^*\re^m\setminus 0$ 
there is $\f\in C_0^\infty(T^*\re^m)$ such that $\f(x_0,\x_0)\neq 0$ and 
\[
\norm{Op^W(\f^h)u}_{L^2} \leq C h^\s, \quad h>0.
\]
Thus, in turn, we learn $u\in B_{2,loc}^{\s,\infty}(\re^{2n})$ if and only if 
for any $(x_0,y_0,\x_0,\y_0)$, $(\x_0,\y_0)\neq (0,0)$, there are 
$\f_1,\f_2\in C_0^\infty(\re^n)$ such that $\f_1(x_0,\x_0)\neq 0$, 
$\f_2(y_0,\y_0)\neq 0$, and 
\[
\norm{Op^W(\f_1^h) U Op^W(\f_2^h)}_{HS} \leq C h^\s, \quad h>0,
\]
where $\norm{\cdot}_{HS}$ denote the Hilbert-Schmidt norm in $B(L^2(\re^n))$. 
Now we choose $\f_3\in C_0^\infty(\re^n)$ so that $\f_3=1$ in a neighborhood of 
$\supp \f_2$. We note 
\begin{align*}
\norm{Op^W(\f_3)}_{HS} &= (2\pi)^{-n/2}\biggpare{\int_{\re^n} |\f_3(x,h\x)|^2dxd\x}^{1/2}\\
&= (2\pi h)^{-n/2} \biggpare{\int_{\re^n} |\f_3(x,\x)|^2dxd\x}^{1/2} 
= Ch^{-n/2}
\end{align*}
for $h>0$ with some $C>0$. Hence we have 
\begin{align*}
\norm{Op^W(\f_1^h) U Op^W(\f_2^h)}_{HS}
&\leq  \norm{Op^W(\f_1^h) U Op^W(\f_2^h)Op^W(\f_3^h)}_{HS}+ R\\
&\leq C h^{-n/2} \norm{Op^W(\f_1^h) U Op^W(\f_2^h)}_{B(L^2)}+R , 
\end{align*}
where 
\[
R=\norm{Op^W(\f_1^h) U Op^W(\f_2^h)(1-Op^W(\f_3^h))}_{HS} =O(h^\infty)
\]
by the symbol calculus. Thus we have proved the following lemma: 

\begin{lem}
If for any $(x_0,y_0,\x_0,\y_0)\in T^*\re^{2n}$ with 
$(\x_0,\y_0)\neq (0,0)$ there are $\f_1,\f_2\in C_0^\infty(T^*\re^n)$ such that 
$\f_1(x_0,\x_0)\neq 0$, $\f_2(y_0,\y_0)\neq 0$ and 
\[
\norm{Op^W(\f_1^h)UOp^W(\f_2^h)}_{B(L^2)}\leq C, \quad h>0, 
\]
then $u\in B_{2,loc}^{-n/2,\infty}(\re^{2n})$. 
\end{lem}

\begin{proof}[Proof of Theorem B.1]
We modify the proof of Theorem~2.1 in \cite{IN3}. 

We first note 
\[
WF(u)\subset \L_S=\bigset{(x,y,\x,-\y)\in T^*\re^{2n}}{(x,\x)=S(y,\y)}.
\]
The proof is almost the same as that of \cite{IN3}. We note if 
$(x_0,y_0,\x_0,-\y_0)\notin \L_S$ with $\y_0\neq 0$, 
it is straightforward to show $(x_0,y_0,\x_0,-\y_0)\notin WF(u)$ 
(as in \cite{IN3}). If $\x_0\neq 0$, we consider $U^*$ and we can also conclude
$(x_0,y_0,\x_0,-\y_0)\notin WF(u)$. 

Now we let $a_1,a_2,\dots, a_N\in S_{cl}^1(\re^n)$ and let $(x_0,\x_0)=S(y_0,\y_0)$. 
We may assume $a_j$ are homogeneous of order one in $\x$-variables. 
By Lemma~B.3 and the proof of Theorem~2.1 of \cite{IN3}, it suffices to show the following 
to conclude $U$ is an FIO of order 0 associated to $S$: There are $\g_1,\g_2\in C_0^\infty(T^*\re^n)$ 
such that $\g_1(x_0,\x_0)\neq 0$, $\g_2(y_0,\y_0)\neq 0$ and 
\begin{equation}\label{est:B-key}
\bignorm{Op^W(\g_1^h)[Ad_S(a_1)\cdots Ad_S(a_N) U] Op^W(\g_2^h)}_{B(L^2)}\leq C,
\quad h\in (0,1]
\end{equation}
with some $C>0$. 

We set $\G_0,\G_1\in C_0^\infty(T^*\re^n)$ so that they are supported in a small 
neighborhood of $(y_0,\y_0)$, $\G_j=1$ on a neighborhood of $(y_0,\y_0)$, and 
$\G_0=1$ on $\supp\G_1$. 
We then set
\[
\f_j(x,\x)= a_j(x,\x)\G_0(x,\x)\in C_0^\infty(T^*\re^n). 
\]
We note, since $a_j$ are homogeneous of order one in $\x$, 
\[
a_j(x,\x)\G_0(x,h\x)= h^{-1}a_j(x,h\x)\G_0(x,h\x)= h^{-1}\f_j(x,h\x).
\]
We also set
\[
\g_1=\G_1\circ S^{-1}, \quad \g_2=\G_1
\]
so that 
\[
\g_1(1-\G_0\circ S^{-1})=0, \quad (1-\G_0)\g_2=0.
\]
These imply, in particular, 
\[
\g_1(x,h\x)(a_j\circ S^{-1})(x,\x) = h^{-1} \g_1(x,h\x)(\f_j\circ S^{-1})(x,h\x),
\]
and 
\[
a_j(y,\y) \g_2(y,h\y)= h^{-1}\f_j(y,h\y) \g_2(y,h\y).
\]
Using these, and applying the $h$-pseudodifferential operator calculus, we learn 
\begin{multline*}
Op^W(\g_1^h)[Ad_S(a_1)\cdots Ad_S(a_N) U]Op^W(\g_2^h) \\
= h^{-N} Op^W(\g_1^h) [Ad_S(\f_1^h)\cdots Ad_S(\f_N^h) U]Op^W(\g_2^h) +O(h^\infty),
\end{multline*}
and this implies the right hand side is bounded by the assumption of Theorem~B.1. 
Now \eqref{est:B-key} follows from this observation, and we conclude the assertion. 
\end{proof}

We note the conditions and the assertion of Theorem~B.1 are microlocal, and hence 
the theorem is easily extended to a statement in a conic set in $T^*\re^n$. 
In the next proposition, we use the extended statement on a conic set. 

\begin{prop}
Let $\re^m=\re^n\times \re^k$, and let $U$ be a bounded operator on $L^2(\re^m)$ and let 
 $S$ be a homogeneous canonical diffeomorphism on $T^*\re^m$. 
 Suppose $U$ commutes with multiplication operators in $y$ so that $U$ is decomposed to
\[
U=\int^\oplus_{\re^k} \tilde U(y) dy \quad \text{on } L^2(\re^m)\cong L^2(\re^k_y,L^2(\re_x^n)),
\]
where $\{U(y)\}$ is a family of operators on $L^2(\re^n_x)$. Suppose also that 
$S$ is decomposed to 
\[
S: (x,\x,y,\y) \mapsto (\tilde S(y)(x,\x), y,\y+g(x,\x,y))
\]
for $(x,\x,y,\y)\in T^*\re^n \cong T^*\re^n_x\times T^*\re^k_y$, 
where $\{\tilde S(y)\}$ is a family of canonical maps on $T^*\re^n_x$. 
If $U$ is an FIO associated to $S$ on a conic set $\bigset{(x,\x,t,\y)}{\x\neq 0}$, 
then for each $y\in\re^k$, $\tilde U(y)$ is an FIO of order 0 associated to $\tilde S(y)$. 
\end{prop}

\begin{rem}
The assumption on $S$ actually follows from the properties of $U$. 
We include it to introduce the notations. 
\end{rem}

\begin{proof}
Let $a\in C_0^\infty(T^*\re^n\setminus 0)$, and let $\f, \g\in C_0^\infty(\re^k)$ such that 
$\f, \g\geq 0$ and $\int \g(\y)d\y =1$. We also denote $\g_z(\y)=\g(\y-z)$ for $z\in\re^k$. 
We consider 
\[
A_z = a_z(x,hD_x,y,hD_y)= a(x,hD_x)\f(y)\g_z(hD_y).
\]
Since $U$ is an FIO, there is $b_z$, which is bounded in $C_0^\infty(T^*\re^m)$ 
uniformly in $h\in (0,1]$,  such that 
\[
UA_z =B_z U +O(h^\infty), \quad B_z=b_z(x,hD_x,y,h D_y)
\]
with the principal symbol: 
\[
a_z\circ S^{-1} = (a\circ \tilde S(y)^{-1})(x,\x)\f(y)\g\bigpare{\y-g(\tilde S(y)^{-1}(x,\x),y)-z}.
\] 
Since $U$ commutes with $\bigset{e^{iy\cdot z}}{z\in\re^k}$, i.e., the translations 
in $\y$-variables, we learn 
\[
b_z(x,\x,y,\y)=b_0(x,\x,y,\y-z),
\]
and the remainder term also satisfies this property. Moreover, 
these symbols decays rapidly outside $S [\supp a_z]$. 

On the other hand, it is easy to see 
\[
\int_{|z|\leq R} A_z dz \to a(x,hD_x)\f(y), \quad 
\int_{|z|\leq R} B_z dz \to \tilde b(x,hD_x,y)
\]
strongly as $R\to\infty$, where 
$\tilde b(x,\x,y)= \int_{\re^k} b_0(x,\x,y,\y)d\y$. The principal symbol of 
$\tilde b$ is given by $(a\circ \tilde S(y)^{-1})(x,\x)\f(y)$. These imply 
\[
\tilde U(y) a(x,hD_x) \f(y)= \tilde b(x,hD_x,y)\tilde U(y) +O(h^\infty), 
\]
where $\tilde b(x,\x,y)-(a\circ \tilde S(y)^{-1})(x,\x)\f(y) =O(h)$. 
Since $\f\in C_0^\infty(\re^k)$ is arbitrary, for a fixed $y\in\re^k$ 
we may replace $\f(y)$ by $1$, and we learn $\tilde U(y)$ is an FIO of order 0
associated to $\tilde S(y)$ by Corollary~B.2
\end{proof}


\end{document}